\documentclass{siamltex}
\usepackage{amsmath,amsfonts,amssymb}
\usepackage{stmaryrd}
\usepackage{latexsym}
\usepackage{amsmath}
\usepackage{epsfig,overpic}
\usepackage{subcaption}
\usepackage{color}

\newcommand{\bI}{\mathbf I}

\newcommand{\bn}{\mathbf n}
\newcommand{\bp}{\mathbf p}
\newcommand{\bq}{\mathbf q}

\newcommand{\bw}{\mathbf w}
\newcommand{\bx}{\mathbf x}
\newcommand{\by}{\mathbf y}

\newcommand{\eoc}[1]{$\text{eoc}_{\texttt{#1}}$}

\newcommand{\R}{\mathbb R}

\newcommand{\wt}{{\bw_T}}
\newcommand{\wte}{{\bw_T^e}}
\newcommand{\G}[2]{\Gamma^{#1}_{#2}}
\newcommand{\norm}[1]{\Vert #1 \Vert}
\usepackage{accents}
\renewcommand*{\dot}[1]{%
   \accentset{\mbox{\large\bfseries .}}{#1}}

\newcommand{\stab}{\rho_n}
\newcommand{\stabold}{\rho_{n-1}}
\newcommand{\consist}{\mathcal{E}_C^n}
\newcommand{\interpol}{\mathcal{E}_I^n}
\newcommand{\err}{\mathbb{E}}
\newcommand{\eps}{\varepsilon}

\newcommand{\V}{\mathcal V}

\newcommand{\T}{\mathcal T}
\newcommand{\Div}{\operatorname{\rm div}}

\newcommand{\nablaG}{\nabla_\Gamma}

\newcommand{\divG}{{\mathop{\,\rm div}}_{\Gamma}}
\newcommand{\gradG}{\nabla_{\Gamma}}

\newcommand{\rr}{\mathbb{R}}

\newcommand{\Gs}{\mathcal{G}} 

\DeclareGraphicsExtensions{.pdf,.eps,.ps,.eps.gz,.ps.gz,.eps.Y}

\newif\ifarxiv

\usepackage{booktabs}
\usepackage{siunitx}
\newcommand{\numbf}[1]{{\underline{\num{#1}}}}

\newcommand{\numQ}[1]{\num[round-precision=2,round-mode=places, scientific-notation=false]{#1}}
\renewcommand{\O}{\mathcal{O}}

\newcounter{ass}

\newcommand{\asslabel}{}
\newcommand{\ifasslabel}[1]{}

\begin{document}
	\title{A finite element method for Allen--Cahn equation on deforming  surface}

\author{
Maxim Olshanskii\thanks{Department of Mathematics, University of Houston, Houston, Texas 77204 (molshan@math.uh.edu).}
\and Xianmin Xu\thanks{ LSEC,ICMSEC,
  NCMIS, Academy of Mathematics and Systems Science, Chinese Academy of Sciences, Beijing 100190, China (xmxu@lsec.cc.ac.cn).
}\and Vladimir Yushutin\thanks{Department of Mathematics, University of Maryland, College Park, Maryland 20742 (yushutin@umd.edu).}
}

\maketitle
\begin{abstract}
The paper studies an Allen--Cahn-type equation defined on a time-dependent  surface as a model of phase separation with order--disorder transition in a thin material layer. By a formal inner-outer expansion, it is shown that the limiting behavior of the solution is a geodesic mean curvature type flow in reference coordinates. A  geometrically unfitted finite element method, known as a trace FEM, is considered for the numerical solution of the equation.  The paper provides  full stability analysis and convergence analysis that accounts for interpolation errors and an approximate recovery of the geometry.
\end{abstract}



\section{Introduction}
Phase separation  may happen in thin material layers such as   polymer films, lipid bilayers, binary alloy interfaces or biophotonic nanostructures. One example of such essentially 2D phenomenon is the lipid rafts formation in a multi-component  plasma membrane, while the membrane is advected by an extracellular fluid  flow and exhibit tangential motion due to the membrane lateral fluidity~\cite{simons1997functional,VEATCH20033074}. In this and some other applications the thin layer is compliant so that a continuum based model represents it  by a surface underdoing radial and lateral deformations.
Motivated by these examples we adopt the model of Allen and Cahn~\cite{allen1979microscopic} to describe the phase evolution on a surface with a prescribed material motion. The model uses a  smooth indicator function $u$ (order parameter) to characterize ordered\,/\,disordered states and a transition region.  This renders the model as a diffusive interface approach.

Before applying a numerical method to the derived Allen--Cahn type equation, the paper addresses well-posedness of the problem and the limiting behaviour of $u$ when the width of the transition region tends to zero. The latter is done here by extending the standard technique of inner (with respect to the transition layer) and outer expansions for the solution. In a steady domain the asymptotic behaviour is well known to be the mean curvature flow~\cite{evans1992phase} for the limit sharp interface  (or the mean geodesic curvature flow for surfaces~\cite{elliott2010modeling}).  In the case of the deforming surface $\Gamma(t)$ we obtain that the for each time $t$ the  material velocity of the sharp interface is defined by (instantaneous) geodesic mean curvature, which can be also seen as a mean curvature type flow in reference coordinates.

The main focus of the paper is a finite element analysis of the  Allen--Cahn type equation posed on an evolving surface.
The paper introduces a geometrically unfitted finite element method, known as a trace FEM~\cite{ORG09,olshanskii2016trace}, to discretize  the problem. The method considers a sharp representation of $\Gamma(t)$ (e.g., as a zero level of a level set  function) and uses degrees of freedom tailored to an ambient tetrahedral mesh, which can be chosen independent of the surface and its evolution.
The numerical approach benefits from the embedding $\Gamma(t)\subset\mathbb{R}^3$  by using tangential calculus to define surface differential operators. Tangential calculus assumes an extension of functions from  $\Gamma(t)$ to its (narrow) neighborhood. The latter is also used here to define a time-stepping numerical procedure following the ideas from~\cite{lehrenfeld2018stabilized,olshanskii2017trace}.
We prove stability and error estimates for the numerical method. The error analysis accounts for all types of discretization errors, e.g., those resulting from the time stepping, polynomial interpolation and the geometric consistency  error due to a possible inexact integration over $\Gamma(t)$. Besides the difficulties associated with time-dependent domains and the treatment of tangential quantities, the current analysis is complicated by the following factor.
While in a stationary domain (e.g., in a non-compliant material surface) the  Allen--Cahn model defines the evolution of the order parameter as the $L^2$-gradient flow of the Ginzburg--Landau energy functional, such minimization property fails to hold for time-dependent domains.

Computational methods and numerical analysis for Allen--Cahn type equations in planar and volumetric  domains  have received much attention in the literature, see  e.g.~\cite{Shen_Yang2010,guillen2014second,liu2015stabilized,hou2017discrete,huang2019adaptive} among recent publications.
At the same time, numerical treatment of surface Allen-Cahn equations is a relatively recent topic in the literature. Work has been done on developing   a closest point finite difference  method \cite{kim2017finite}, a mesh free method \cite{mohammadi2019numerical},  and finite elements methods (FEMs) \cite{dziuk2007surface,elliott2010modeling,yushutin2019computational,xiao2020unconditionally} as the most versatile and mathematically sound approach. Among those papers \cite{elliott2010modeling} allows deformation of the surface due to line tension forces and applies a (fitted) FEM  on a triangulated surface. The authors of \cite{yushutin2019computational}  applied unfitted (trace) FEM to phase-field models on stationary surfaces. Numerical analysis for equations governing phase separation on (evolving) surfaces is largely an open topic. Another two closely related studies  \cite{elliott2015evolving,yushutin2020numerical} deal with FEMs for the Cahn--Hilliard equation on a time-dependent surface: in~\cite{elliott2015evolving} the authors develop numerical analysis of a fitted FEM  and \cite{yushutin2020numerical} applies the trace FEM. {\color{black} Trace FEM is a member of a large family of geometrically unfitted finite element methods along with such as XFEM~\cite{moes1999finite}, immersed interface FEM~\cite{li1998immersed} and cutFEM~\cite{cutFEM}, the most closely related approach.
For the later approach, other authors considered stabilized  space–time formulations~\cite{grande2014eulerian,hansbo2016cut} and semi-Lagrangian type methods~\cite{hansbo2015characteristic} to integrate in time PDEs posed on evolving surfaces.}

The outline of the paper is as follows. In Section \ref{s:form}, we introduce the model.  The weak formulation of the problem and its well-posedness are discussed in Section \ref{s:weak}. An asymptotic behaviour of the solution to the problem is studied in Section \ref{s:asympt}. After necessary preliminaries, the numerical method is introduced in Section \ref{s:discretization}. Error and stability analyses are carried out in Section \ref{s:Analysis}. Section \ref{s:Numerics} supplements the paper with numerical examples.

\section{Allen--Cahn equation on an evolving surface}\label{s:form}
Consider a material surface $\Gamma(t)\subset\mathbb{R}^3$, $t\in[0,T]$, with density distribution $\rho:\Gamma(t)\to \mathbb{R}$. Assume $\Gamma(t)$ is passively advected by a smooth velocity field $\bw=\bw(\bx,t)$, $\bx\in\R^3$, and for all times $\Gamma(t)$ stays smooth, closed  ($\partial \Gamma(t) =\emptyset$), connected and orientable. We are interested in a phase separation process on $\Gamma(t)$ with a transition between order and disorder states. The state of matter at $\bx\in\Gamma(t)$ is characterized by a smooth indicator function $u(\bx,t)$, $u\,:\,\Gamma(t)\to [-1, 1]$, with $u\simeq -1$ in the less ordered phase and $u\simeq 1$ in the more ordered phase.

To describe an evolution of phases, we follow the classical approach of Allen and Cahn~\cite{allen1979microscopic}  and assume that an instantaneous change in the order per  area $s(t)\subset \Gamma(t)$ is proportional to  the variation of the  total specific free energy for $s(t)$:
\begin{equation}\label{law}
\frac{d}{dt}\int_{s(t)}\rho u\,ds = -\int_{s(t)}c_k\frac{\delta e(u)}{\delta u}\,ds,
\end{equation}
where $c_k$ is a positive kinetic coefficient, and the  energy density is given by
\[
e(u)=\rho\left(\frac{1}{\epsilon^{2}} \\
F(u) +|\nabla_\Gamma u|^2\right),
\]
where $\nabla_\Gamma u$ is the tangential gradient of $u$ ({\color{black} cf. definition in \eqref{aux259}}).
The energy of a homogeneous state $F(u)$ has a double--well form of Ginzburg--Landau potential to allow for phase separation, and $\epsilon$ is a characteristic width of a transition region between phases. Further we choose $F(u)=(1-u^2)^2/4$.

Application of the surface Reynolds transport theorem (also known as the Leibniz formula for evolving surfaces, e.g., \cite{elliott2010surface}) to \eqref{law} gives
\begin{equation*}
\int_{s(t)}(\dot{(\rho u)} + \rho u \divG \bw)\,ds = -\int_{s(t)}c_k\frac{\delta e(u)}{\delta u}\,ds.
\end{equation*}
By $\dot{f}$ we denote the material derivative of a smooth  function $f$ defined on $\Gamma(t)$ for $t\in[0,T]$ and $\divG$ stands for the surface divergence ({\color{black} cf. \eqref{aux259}}).
Computing the functional derivative of $F(u)$ with respect to $u$, $f(u)=F'(u)$, and varying $s(t)$ for any fixed $\Gamma(t)$, $t\in[0,T]$ leads to the Allen--Cahn equation on the deforming surface:
\begin{equation}\label{Allen1}
\dot{(\rho u)} + \rho u \divG \bw = -\rho c_k({\epsilon}^{-2}f(u) -\divG(\rho\gradG{}u))\quad\text{on}~\Gamma(t),~t\in(0,T).
\end{equation}
Likewise, the conservation of mass and the surface Reynolds transport theorem yield the identity
\begin{equation}\label{mass}
\dot{\rho} + \rho \divG \bw=0\quad\text{on}~\Gamma(t).
\end{equation}
Thanks to \eqref{mass}, the surface  Allen--Cahn equation \eqref{Allen1} can be written in the equivalent form
\begin{equation}\label{Allen2}
\dot{u} = -c_k\left(\epsilon^{-2}f(u) -\frac{1}{\rho}\divG(\rho\gradG{}u)\right)\quad\text{on}~\Gamma(t),~t\in(0,T).
\end{equation}
The equation should be complemented with the initial condition $u(\bx, 0)= u_0(\bx)$, $\bx\in\Gamma(0)$, describing the state of matter at time $t=0$.

Equations  \eqref{Allen1} or \eqref{Allen2} are solved for the order parameter $u$ with given $\rho$ satisfying \eqref{mass}.
In this paper, we assume $\rho=\mbox{const}$. In practice, this assumption is plausible for surfaces with initially homogeneous density distribution and exhibiting small  or area-preserving deformations.
The latter is characterised by $\divG \bw=0$ and is a valid assumption for several types of biological membranes, such as lipid mono- or bi-layers~\cite{lipowsky1991conformation,seifert1997configurations}.   Due to this assumption, the model (slightly) simplifies to the following system of equation and initial condition:
\begin{equation}\label{transport}
\left\{
\begin{split}
   \dot{u}& = -{	\epsilon^{-2}}f(u) +\Delta_\Gamma u\quad\text{on}~\Gamma(t),~t\in(0,T), \\
     u& =u_0\quad\text{on}~\Gamma(0),
\end{split}
\right.
\end{equation}
 $\Delta_\Gamma$ is the Laplace--Beltrami operator and we set $c_k=1$. 

We close this section by noting the analogy between Allen--Cahn equations \eqref{Allen1} or \eqref{Allen2} and  those
describing the compressible two-phase fluid flow  (in the Euclidean space)  with phase transition; see~\cite{blesgen1999generalization}.

\subsection{Preliminaries}\label{s:prelim}
We need more precise assumptions for the evolution of $\Gamma(t)$.  To formulate them, assume that $\bw$ and  $\Gamma_0$ are sufficiently smooth such that for all $y \in \Gamma_0$ the ODE system
\[
  \Phi(y,0)=y, \quad \frac{\partial \Phi}{\partial t}(y,t)= \bw(\Phi(y,t),t), \quad t\in [0,T],
\]
has a unique solution $x:=\Phi(y,t) \in \Gamma(t)$, which defines the Langrangian mapping $\Phi:\Gamma_0\to \Gamma(t)$. The  inverse mapping is given by $\Phi^{-1}(x,t):=y \in \Gamma_0$, $x \in \Gamma(t)$. With the help of  $\Phi$, we define the bijection $\Psi$ between $\Gamma_0 \times [0,T]$, with $\Gamma_0:=\Gamma(0)$,and the space-time manifold
\[\Gs: = \bigcup\limits_{t \in (0,T)} \Gamma(t) \times \{t\},\quad  \Gs\subset \Bbb{R}^{4} \]
as follows
\begin{equation} \label{defF}
 \Psi:\,   \Gamma_0 \times [0,T] \to \Gs, \quad ~\Psi(y,t):=(\Phi(y,t),t).
\end{equation}
\emph{We assume $\Psi$ is a $C^2$-diffeomorphism between these manifolds}.

For $\Gamma(t)$,  consider a signed distance function $\phi(t)$ (positive in the exterior and negative in the interior of $\Gamma(t)$). Let $\O_\delta(\Gs)$ be a tubulate $\delta$-neighborhood of $\Gamma$:
\[
\O_\delta(\Gs):=\{(x,t)\in\mathbb{R}^4\,:\,|\phi(x,t)|\le\delta\}.
\]
 The above assumptions imply that for sufficiently small $\delta>0$ it holds $\phi\in C^2(\O_\delta(\Gs))$ and  the normal projection onto $\Gamma(t)$, $\bp:\O_\delta(\Gs)\to \Gamma(t)$ is well defined for each $t\in[0,T]$. We fix such $\delta$ and further often skip it in notation $\O(\Gs)=\O_\delta(\Gs)$. Likewise, we shall write $\O_\delta(\Gamma(t))$ to denote a $\delta$-neighborhood of $\Gamma(t)$ in $\mathbb{R}^3$ and $\O(\Gamma(t))=\O_\delta(\Gamma(t))$ for $\delta$ as above. For every fixed $t\in[0,T]$, the gradient of $\phi$ defines in $\O(\Gamma(t))$ normal direction to $\Gamma(t)$ with $\bn=\nabla \phi$ being the outward normal vector on $\Gamma(t)$, {\color{black}here and below $\nabla$ is spacial gradient in $\mathbb{R}^3$}.

For a smooth $u$ defined on $\Gs$, a function $u^e$ denotes the extension of $u$ to $\O(\Gs)$ along spatial normal directions to the level-sets of $\phi$, it holds
$\nabla u^e\cdot\nabla \phi=0$ in $\O(\Gs)$, $u^e=u$ on $\Gs$, and $u^e(\bx)=u^e(\bp(\bx))$ in $\O(\Gs)$. The extension $u^e$ is smooth once $\phi$  and $u$ are both smooth.
 Further, we use the same notation $u$ for the function on $\Gs$ and its extension to  $\O(\Gs)$.

Once a function $u$ on
  $\Gs$ is identified with its extension on $\O(\Gs)$, one can write the surface differential operators arising in the model, in terms of tangential calculus:
  \begin{equation}\label{aux259}
  \nabla_\Gamma u =(\bI-\bn\times\bn^T)\nabla u,\quad {\Div}_\Gamma\bw=\mbox{tr}\left((\bI-\bn\times\bn^T)\nabla \bw\right),\quad
  \Delta_\Gamma u= {\Div}_\Gamma\nabla_\Gamma u.
  \end{equation}
Furthermore,  one can expand the intrinsic surface quantity $\dot{u}$ in Eulerian terms:
\begin{equation}\label{aux264}
\dot{u}= \frac{\partial u}{\partial t} + \bw \cdot \nabla u.
\end{equation}
Identity \eqref{aux264} allows us to rewrite  \eqref{transport} as follows:
\begin{equation}
\left\{
\begin{split}
\frac{\partial u }{\partial t}+\bw\cdot\nabla u&=-\epsilon^{-2}f(u)+\Delta_{\Gamma} u \quad\text{on}~~\Gamma(t),\\
\nabla u\cdot\nabla\phi&=0 \quad\text{in}~~\O(\Gamma(t))
\end{split}~~t\in (0,T],
\right.
\label{transport1}
\end{equation}
subject to  $u =u_0$ {on} $\Gamma(0)$.
This formulation will be useful for the design of a finite element method in Section~\ref{s:discretization}.
We note that  equalities \eqref{aux259}--\eqref{aux264} are  valid for any smooth extension (not necessarily a normal one).

\section{Asymptotic analysis} \label{s:asympt}
In this section, we  study an asymptotic behaviour of $u$ solving \eqref{transport}
when $\eps$ goes to zero. Our analysis follows the inner-outer expansion arguments, which are now standard for phase-field equations defined on Euclidean domains in $\mathbb{R}^d$, $d=2,3$, \cite{caginalp1986analysis,caginalp1988dynamics,pego1989front} and also has been used recently to study  sharp interface limits of two phase-field models  defined on surfaces~\cite{elliott2010surface,oConnor2016cahn}.

We assume $t\ge t_0$ sufficiently large 
such that the separation of phases happened and  $u$ exhibits an inner layer (diffuse interface) of width $O(\eps)$. Consider the  central line  of the diffuse interface defined as the zero level of $u$, $\gamma(t):=\{\bx\in\Gamma(t)\, :\, u(\bx,t)=0\}$. 
 For all $t\in(t_0,T)$ we assume that $\gamma(t)$ is a smooth closed curve on $\Gamma(t)$. The interior and exterior  domains with respect to $\gamma(t)$ are denoted by  $\Gamma^\pm(t):=\{\bx\in\Gamma(t)\, :\, \pm u(\bx,t)>0\}$. 

{\it Outer expansion.} Denote by $u^{\pm}$ the order parameter restricted to $\Gamma^{\pm}$.
Following, e.g.,  \cite{caginalp1988dynamics} we assume that away from the interfacial layer around  $\gamma(t)$, both $u^{\pm}$
can be expanded in the form
\begin{equation}\label{outer}
u^{\pm}(\bx,t)=u^{\pm}_0(\bx,t)+\eps u^{\pm}_1(\bx,t)+\cdots,\quad
\end{equation}
with smooth $u^{\pm}_k(\bx,t)$.
 Substituting \eqref{outer}  into \eqref{transport} and using
the Taylor expansion for $f(u)$,
$
f(u^{\pm})=f(u_0^{\pm})+\eps f'(u_0^{\pm}) u^{\pm}_1+\cdots,
$
yield
\begin{align*}
&\big(\dot u_0^{\pm}(\bx,t)+\eps \dot u_1^{\pm}(\bx,t)+\cdots\big)-\Delta_{\Gamma} \big(u_0^{\pm}(\bx,t)+\eps u_1^{\pm}(\bx,t)+\cdots\big) +\eps^{-2} \big(f(u_0^{\pm})+\eps f'(u_0^{\pm}) u_1+\cdots\big)=0. 
\end{align*}
Considering the leading order term with respect to $\eps\to 0$  gives
$f(u_0^{\pm})=0.$
Therefore, away from the  layer it holds
\begin{equation}\label{e:outer0}
u_0^{\pm}(\bx,t)=\pm 1.
\end{equation}

{\it Inner expansion.} 
Denote by $d_\gamma$  the signed geodesic distance on $\Gamma(t)$ for any fixed $t$, and $\pm d_{\gamma}(\bx)>0$ for $\bx\in \Gamma^\pm$.
Consider the inner layer $U_\eps(\gamma(t))$, which we define as an $O(\eps)$ neighborhood of $\gamma(t)$:
$U_\eps(\gamma(t)):=\{\bx\in \Gamma(t)\,:\,|d_\gamma(\bx)|\le c_0\,\eps\}$, with sufficiently large $c_0$, independent of $\eps$. We assume $\eps$ to be sufficiently small such that the geodesic closest point projection $\bq(\bx):\,U_\eps(\gamma(t))\to\gamma(t)$ is well-defined so that $(\bq(\bx),d_{\gamma}(\bx))$ is the local (time dependent) coordinate system in $U_\eps(\gamma(t))$.  In $U_\eps(\gamma(t))$ the conormal directions are defined by the tangential vector field  $\mathbf{m}=\nabla_{\Gamma}d_{\gamma}$. For $\bx\in\gamma(t)$, $\mathbf{m}(\bx)$ is a unit conormal of $\gamma(t)$  pointing into $\Gamma^+(t)$.

Following \cite{caginalp1988dynamics,pego1989front}, we introduce a fast variable in $U_\eps(\gamma(t))$ by re-scaling the coordinate in the conormal direction $\xi=\frac{d_{\gamma}(\bx)}{\eps}$, and represent  $u(\bx,t)$ as
\begin{equation}
u(\bx,t)=\tilde{u}(\bx,\xi,t)\quad\text{for}~\bx\in U_\eps(\gamma(t)),
\end{equation}
where $\tilde{u}(t):\,U_\eps(\gamma(t))\times (-c_0,c_0)\to\mathbb{R}$ is 
defined as $\tilde{u}(\by,\xi,t):=u(\bx,t)$ for $\bx\in U_\eps(\gamma(t))$ such that $\bq(\by)=\bq(\bx)$ and $\xi=d_{\gamma}(\bx)/{\color{black}\eps}$.
Given the new variables  we find  the identities:
\begin{equation}\label{aux353}
\nabla_{\Gamma} u=\nabla^{\bx}_{\Gamma} \tilde{u}+\eps^{-1} \partial_{\xi}\tilde{u}\,\mathbf{m}, \quad
\Delta_{\Gamma} u=\Delta^{\bx}_{\Gamma} \tilde{u}
+\eps^{-2}\partial_{\xi\xi}\tilde{u}+\eps^{-1}\partial_{\xi}\tilde{u}\Delta_{\Gamma} d_{\gamma},
\end{equation}
where for the second equality we used   $\mathbf{m}\cdot\nabla^{\bx}_{\Gamma}  \tilde{u}=0$ and $\mbox{div}_\Gamma\mathbf{m}=\Delta_{\Gamma} d_{\gamma}$ {\color{black}(same identities in terms of fast and slow surface variables are deduced by slightly different arguments in~\cite{oConnor2016cahn} and \cite{garcke2016coupled})}.
Denoting by $\dot{\tilde{u}}$ the material derivative of $\tilde{u}(\by,\xi,t)$ 
we also compute
\begin{equation}\label{aux366}
\dot{{u}}=\dot{\tilde{u}}+\eps^{-1}\partial_{\xi}\tilde{u}\,\dot{d_\gamma}.
\end{equation}

We assume that $\tilde{u}$ in the layer can be expanded
$$
\tilde{u}(\bx,\xi,t)=\tilde{u}_0(\bx,\xi,t)+\eps\tilde{u}_1(\bx,\xi,t)+\cdots,
$$
with smooth $\tilde{u}_0,\tilde{u}_1,\dots$.
Substituting this in \eqref{transport}, using \eqref{aux353}--\eqref{aux366}  and  Taylor expansion for $f(\tilde{u})$,
i.e. $f(\tilde{u})=f(\tilde{u}_0)+\eps f'(\tilde{u}_0)\tilde{u}_1+\cdots $,
we find that $O(\eps^{-2})$ order terms give
\begin{equation}\label{e:inner0}
-\partial_{\xi\xi}\tilde{u}_0+f(\tilde{u}_0)=0.
\end{equation}
Accounting for  $O(\eps^{-1})$ order terms we obtain
\begin{equation}\label{e:inner1}
\partial_{\xi} \tilde{u}_0\left(\dot{d_{\gamma}} -\Delta_{\Gamma} d_{\gamma}\right)
-2\nabla_{\Gamma} d_{\gamma}\cdot\nabla_{\Gamma}(\partial_{\xi}\tilde{u}_0)-\partial_{\xi\xi}\tilde{u}_1+f'(\tilde{u}_0)\tilde{u}_1=0.
\end{equation}
To proceed we need conditions on $\tilde{u}_0$ for $\xi\to \infty$ (which can be allowed if $\eps\to0$).

{\it Matching conditions.} We now have a representation of the solution in the narrow layer around $\gamma(t)$ and another representation valid away from the interface. Following  \cite{caginalp1988dynamics,pego1989front} we consider  matching conditions between these two representations. We formulate the conditions below, while details of derivation can be found in  \cite{garcke2006second}.
Denote $u^{\pm}_k(\bx,t)=\lim\limits_{s\rightarrow\pm 0}u^{\pm}_k(\bx+s\mathbf{m},t)$ when $\mathbf{x}\in\gamma(t)$ and $\eps\to0$,  and similar we define $\nabla_{\Gamma} u^{\pm}_0(\bx,t)$ for  $\mathbf{x}\in\gamma(t)$. The matching conditions read:
\begin{align}
 \tilde{u}_0(\bx,\xi,t)&=u_0^{\pm}(\bx,t), && \hbox{as }\xi\rightarrow\pm \infty,~\eps\xi\rightarrow0 \label{e:match1}\\
\tilde{u}_1(\bx,\xi,t)&=u_1^{\pm}(\bx,t)+\xi \mathbf{m}\cdot\nabla_{\Gamma}u_0^{\pm}(\bx,t), &&\hbox{as }\xi\rightarrow\pm \infty,~\eps\xi\rightarrow0  \label{e:match2}\\
\partial_{\xi}\tilde{u}_1(\bx,\xi,t)&=\mathbf{m}\cdot\nabla_{\Gamma}u_0^{\pm}(\bx,t), &&\hbox{as }\xi\rightarrow\pm \infty,~\eps\xi\rightarrow0 .  \label{e:match3}
\end{align}

From condition~\eqref{e:match1} and  \eqref{e:outer0} it follows that
\begin{equation}\label{aux417}
\lim_{\xi\rightarrow\pm\infty}\tilde{u}_0=\pm 1.
\end{equation}
This and  $\tilde{u}_0(\bx,0,t)=0$ supplies the equation~\eqref{e:inner0} with necessary boundary conditions. For $f(\tilde{u}_0)=-\tilde{u}_0+\tilde{u}_0^3$ it provides us with the unique  solution
\begin{equation*}
\tilde{u}_0(\bx,\xi,t)=\tanh(\xi/\sqrt{2}).
\end{equation*}
In particular, we see that $\tilde{u}_0$ does not depend on $(\bx,t)$. This simplifies  equation  \eqref{e:inner1}  to
\begin{equation*}
\partial_{\xi} \tilde{u}_0\left(\dot{d_{\gamma}} -\Delta_{\Gamma} d_{\gamma}\right)
-\partial_{\xi\xi}\tilde{u}_1+f'(\tilde{u}_0)\tilde{u}_1=0.
\end{equation*}
We multiply the above identity by $\partial_\xi \tilde{u}_0$  and  integrate it for $\xi\in(-\infty,\infty)$.
This leads to
\begin{equation}\label{e:inner1n}
\sigma \left(\dot{d_{\gamma}}-\Delta_{\Gamma} d_{\gamma}\right)
-\int_{-\infty}^{\infty}\!\!(\partial_{\xi}\tilde{u}_0)[\partial_{\xi\xi}\tilde{u}_1-f'(\tilde{u}_0)\tilde{u}_1]d\xi=0.
\end{equation}
where $\sigma:=\int_{-\infty}^{\infty}(\partial_\xi \tilde{u}_0)^2 d\xi$ is a positive constant that can be interpreted as  interface tension coefficient. Now let us take a further look into matching conditions \eqref{e:match1}--\eqref{e:match3}. The first one  implies $f'(\tilde{u}_0)=f(\tilde{u}_0)=0$ for $\xi\to\pm\infty$. Since $\mathbf{m}\cdot\nabla_{\Gamma}u_0^{\pm}(\bx,t)=0$, from  \eqref{e:match2} and \eqref{e:match3} we also see that $|\tilde{u}_1|$ is bounded and  $\partial_{\xi}\tilde{u}_1=0$ for $\xi\to\pm\infty$. Using these limit values for the integration by parts, we obtain
\[
\int_{-\infty}^{\infty}\!\!\partial_{\xi}\tilde{u}_0[\partial_{\xi\xi}\tilde{u}_1-f'(\tilde{u}_0)\tilde{u}_1]d\xi
=\int_{-\infty}^{\infty}\!\!\partial_{\xi}\tilde{u}_0\partial_{\xi\xi}\tilde{u}_1-\partial_{\xi}f(\tilde{u}_0)\tilde{u}_1 d\xi =\int_{-\infty}^{\infty}\!\![\partial_{\xi\xi} \tilde{u}_0-f(\tilde{u}_0)]\partial_{\xi}\tilde{u}_1d\xi =0,
\]
where for the last equality we use~\eqref{e:inner0}.
Equation~\eqref{e:inner1n} reduces to
\begin{equation}\label{e:dynDistn}
\dot{d_{\gamma}}-\Delta_{\Gamma} d_{\gamma}=0.
\end{equation}
Consider the limiting interface $\gamma(t)$ as the zero level of the order-parameter as $\eps\to0$.
Equation \eqref{e:dynDistn} for the signed distance function  describes the dynamics of $\gamma(t)$ on the passively evolving material surface $\Gamma(t)$.
The quantity  $\Delta_{\Gamma} d_{\gamma}=\kappa_g$ is
the geodesic curvature of $\gamma(t)$ on $\Gamma(t)$ satisfying that $\kappa_g(\bx)$ is positive when $\Gamma(t)^{-}$ is convex
at $\bx$. While $\dot{d_{\gamma}}=0$ corresponds to the passive evolution along material trajectories, $\dot{d_{\gamma}}=\kappa_g$ can be seen as an active evolution or a mean curvature type flow in the reference (Lagrangian) coordinates.  The (tangential) geometric evolution of the sharp interface is defined by the conormal velocity of $\gamma(t)$ given by  $\mathbf{m}\cdot\mathbf{w}-\kappa_g$.

\section{Weak formulation and well-posedness}\label{s:weak}
Consider a slightly more general problem:
\begin{align}\label{transport_ST}
\dot{u}  +\alpha u- \Delta_{\Gamma} u +\eps^{-2}{f(u)}=0\quad\text{on}~~\Gamma(t),\\
u(\bx,0)=u_0 \quad\text{for}~~\bx \in \Gamma(0),
\end{align}
with an $L^\infty(\Gs)$ function $\alpha$, and let $\alpha_{\infty}:=\|\alpha\|_{L^\infty(\Gs)}$. Following \cite{Shen_Yang2010} we consider a modified double-well potential ${F}$ such that for some $M > 1$
  	\begin{align}
 	{F}'(x)={f}(x)=
 \begin{cases}
({3M^2-1})x-2M^3,\quad x>M,\\
 x(x^2-1), \qquad\qquad \quad x\in[-M,M],\\
({3M^2-1})x+2M^3,\quad x<-M.
 \end{cases}
 	 \end{align}
Function $f(x)$ satisfies the following growth conditions  with $L=3M^2-1$
	\begin{align}\label{growth}
	|{f}(x)|\leq L |x|,\quad f(x)x\ge -x,
	\end{align}
and Lipschitz  condition:
		\begin{align}\label{variation}-1\leq\frac{{f}(x)-{f}(y)}{ x-y}\leq{}L,\quad\forall x,y \in \mathbb{R},~  x\neq y.
		\end{align}

Given our assumptions on the evolution of $\Gamma(t)$, the scalar product
$$
(u,v)_0=\int_0^T\int_{\Gamma(t)} u v \,ds dt
$$
induces a norm $\|\cdot\|_0$ on $L^2(\Gs)$ equivalent to the standard $L^2(\Gs)$-norm.
Besides standard  Lebesgue spaces $L^q(\Gs)$, $1\le q\le\infty$, and Sobolev spaces $H^k(\Gs)$, $k=1,2,\dots$, we need the following analogues of standard Bochner spaces:
\begin{align*}
H&=\{u\in L^2(\Gs)\,:\, \|\nablaG u\|_{0}<\infty\},\quad\text{with}~ (u,v)_H=(u,v)_{0}+(\nablaG u, \nablaG v)_{0},\\
L^{\infty}_{1} &=\{u\in L^\infty(\Gs)\,:\, \mbox{ess}\sup\limits_{t\in[0,T]}\|\nablaG u\|_{L^2(\Gamma(t))}\le\infty\},\\
W&=\{u\in L^{\infty}_{1} \,:\, \dot{u} \in L^2(\Gs)\},\quad   \|u\|^2_W =\|u\|^2_{L^{\infty}_{1}}+\|\dot{u}\|^2_{L^2(\Gs)}.
\end{align*}
From~\cite{olshanskii2014eulerian,elliott2015evolving} we know that $H$ is a Hilbert space and smooth functions are everywhere dense in $H$ and $W$.

Exploiting the smoothness  properties of the mapping $\Psi$ between $\Gamma_0\times(0,T)$ and $\Gs$ one shows (cf. \cite{olshanskii2014eulerian,elliott2015evolving}) that the following isomorphisms hold algebraically and topologically: $H\cong L^2(0,T; H^1(\Gamma{}_0))$ 
and $W\cong L^\infty(0,T; H^1(\Gamma{}_0))\cap H^1(0,T;L^2(\Gamma{}_0))$.

We consider the following weak formulation of \eqref{transport_ST}: For $u_0\in H^1(\Gamma_0)$, find $u\in {W}$ such that $u(0)=u_0$ and
\begin{equation}\label{weak}
  (\dot{u},v)_0  +(\alpha u+\eps^{-2}{f(u)},v)_0+(\nabla_\Gamma u,\nabla_\Gamma v)_0=0,\quad\text{for all } v\in H.
\end{equation}

\begin{lemma}\label{Lem:wp} The week formulation \eqref{weak} is well posed.
\end{lemma}
\begin{proof} A standard approach to the analysis of Allen-Cahn type equations solvability  is based on the energy minimization principle, which does not hold in the case of equations posed in the evolving domain.
Hence we consider a different argument. For $\hat{u}=u\circ\Psi\in L^\infty(0,T; H^1(\Gamma{}_0))\cap H^1(0,T;L^2(\Gamma{}_0))$, $v\in L^2(0,T; H^1(\Gamma{}_0))$ we rewrite \eqref{weak} in the reference cylinder $\widehat{S}=\Gamma{}_0\times (0,T)$:
\begin{equation}\label{weak_ref}
  \int_0^T\int_{\Gamma_0}\left\{(\hat{u}_t+\alpha \hat{u}+\eps^{-2}f(\hat{u}))\hat{v}+(\nablaG F)^{-T}\nabla_\Gamma \hat{u}:(\nablaG F)^{-T}\nabla_\Gamma \hat{v}\right\}\mu d\hat{s}dt=0,
\end{equation}
for all $\hat{v}\in L^2(0,T; H^1(\Gamma{}_0))$. Here $\mu\in C^1(\overline{\widehat{S}})$, $\nablaG \Psi\in C^1(\overline{\widehat{S}})^{3\times 3}$ are such that $\mu>0$ and $\nablaG \Psi(\nablaG\Psi)^{T}$ is uniformly bounded on $\widehat{S}$. Therefore, the problem \eqref{weak_ref} can be formulated to fit an abstract framework from~\cite{schimperna2000abstract}:
Find $\hat{u}\in L^\infty(0,T; H^1(\Gamma{}_0))\cap H^1(0,T;L^2(\Gamma{}_0))$ such that $\hat{u}(0)=u_0\in  H^1(\Gamma{}_0)$ and
\begin{equation}\label{aux507}
M\hat{u}'+ B\hat{u}+\gamma(\hat{u})=0\quad\text{in}~H^{-1}(\Gamma{}_0)\quad\text{for a.e.}~t\in[0,T],
\end{equation}
where operators $M\in\mathcal{L}(L^2(\Gamma_0))$, $B\in\mathcal{L}(H^{1}(\Gamma{}_0),H^{-1}(\Gamma{}_0))$  and $\gamma: H^{1}(\Gamma{}_0)\to L^2(\Gamma{}_0)$ are defined by the identities
\begin{align*}
 M \hat{w}=\mu\hat{w},\quad
\langle B\hat{u},\hat v\rangle=\int_{\Gamma_0}(\nablaG \Psi)^{-T}\nabla_\Gamma \hat{u}:(\nablaG \Psi)^{-T}\nabla_\Gamma \hat{v}\mu d \hat{s},\quad
\gamma(\hat{u}) =(\alpha \hat{u}+\eps^{-2}f(\hat{u}))\mu,
\end{align*}
for all $\hat{w}\in L^2(\Gamma_0)$, $ \hat{u},  \hat{v}\in H^{1}(\Gamma{}_0)$.
It is easy to verify that $M$ is positive definite, $B$ is  such that
\begin{equation}\label{aux511}
\langle B v,v\rangle\ge \hat c_1\|v\|^2_{H^{1}(\Gamma{}_0)}-\hat c_2\|v\|^2_{L^2(\Gamma{}_0)},\quad\text{with some}~ \hat c_1>0,~\hat c_2\ge0,
\end{equation}
and $\gamma$ is continuous and, {thanks to \eqref{growth}},
\begin{equation}\label{aux515}
\begin{split}
\|\gamma(v)\|^2_{L^2(\Gamma{}_0)}&\le \hat C_1+\hat C_2\|v\|^2_{H^{1}(\Gamma{}_0)},\quad\text{with some}~ \hat C_1>0,~\hat C_2>0,\\
(\gamma(v)-\gamma(w),v-w)_{L^2(\Gamma{}_0)}&\ge -\hat C_0\|v-w\|^2_{L^2(\Gamma{}_0)},\quad\text{with some}~ \hat C_0>0,
\end{split}
\end{equation}
for all $v,w\in  H^1(\Gamma{}_0)$.

Problem~\eqref{aux507}--\eqref{aux515} is well posed (\cite[Theorem~2.1]{schimperna2000abstract}) and so is \eqref{weak}.
\end{proof}

\section{Discretization method} \label{s:discretization}
To set up a numerical method, one needs to define a time-stepping procedure, spatial discretization approach and a practical way of handling surface integrals and derivatives. The approach taken here benefits from the embedding of $\Gamma(t)$ in $\mathbb{R}^3$ for all $t\in[0,T]$, which allows to use tangential calculus in an ambient (bulk) functional space (rather than computations in intrinsic time-dependent surface coordinates). The bulk space supports well-defined traces of functions on $\Gamma(t)$ and functions from the bulk space are further approximated in a standard time-independent finite element space. Our time-stepping procedure exploits an observation made earlier in section~\ref{s:prelim} that a function on $\Gamma(t)$ can be identified with its smooth extension to a neighborhood of the surface. Finally, the geometry representation is based on the implicit definition of $\Gamma_h(t)$, an approximation of $\Gamma(t)$, as a zero level of a  finite element function.
Altogether, this approach resembles the trace finite element method for partial differential equations on evolving surfaces introduced and analyzed in \cite{olshanskii2017trace,lehrenfeld2018stabilized} for the diffusion problem on $\Gamma(t)$. The approach is also known as a hybrid  FD in time -- trace FEM in space, since a (standard) finite difference scheme is adopted for treating the time dependence and an unfitted finite element method is used in space.

We start with explaining the time-stepping method.

\subsection{Time-stepping scheme}
Consider a uniformly distributed time nodes  $t_n=n\Delta t$, $n=0,\dots,N$, with the uniform time step $\Delta t=T/N$.
It is crucial to assume that $\Delta t$ is sufficiently small that
\begin{equation}\label{ass1}
  \Gamma(t_n)\subset\O( \Gamma(t_{n-1}))\quad~n=1,\dots,N. ~ 
\end{equation}
Recall that $\O( \Gamma(t))$ is a neighborhood of the surface, where the normal projection on $\Gamma(t)$ is well defined, and so are the extensions of surface quantities.

Using the notation $u^n$ for an approximation to $u(t_n)$, and $\phi^n = \phi(t_n)$, we consider the following  semi-implicit first order method  for the Eulerian formulation \eqref{transport1} of the Allen-Cahn surface problem:
\begin{equation}
\left\{
\begin{split}
(1 + \beta_s \Delta t)\frac{u^n-u^{n-1}}{\Delta t}+\bw^n\cdot\nabla u^n 
- \Delta_{\Gamma} u^n&=-\eps^{-2} f(u^{n-1})\quad\text{on}~~ \Gamma(t_n),\\
\nabla u^{n}\cdot\nabla\phi^n&=0\quad\text{in}~~\O(\Gamma(t_n)).
\end{split}\right.
\label{e:ImEuler}
\end{equation}
Here $\beta_s>0$ is a stabilization parameter as suggested in \cite{Shen_Yang2010} to allow the explicit treatment of the non-linear part on the right-hand side of \eqref{e:ImEuler}. This leads to  a linear problem with respect to  $u^n$ on each time step. More important is that the function $u^{n-1}$ is well-defined on $\Gamma(t_n)$ through its extension. Indeed, if one considers \eqref{e:ImEuler} with index shifted $n\to n-1$, i.e. eq. \eqref{e:ImEuler} written for the previous time step, then the second equation defines the extension of $u^{n-1}$ to $\O( \Gamma(t_{n-1}))$ and because of \eqref{ass1} it defines an extension to $\Gamma(t_n)\subset\O( \Gamma(t_{n-1}))$. Therefore, all terms in \eqref{e:ImEuler} on the current step  are well defined.

For a finite element method, we shall need the integral formulation of \eqref{e:ImEuler}, where we enforce the second equation weakly, as a constraint: Any  smooth $u^{n}$ solving \eqref{e:ImEuler} satisfies
\begin{multline} \label{e:conttraceFEM1}
\int_{\Gamma(t_n)}\!\!\big( (1 + \beta_s \Delta t) \frac{u^n-u^{n-1}}{\Delta t}+\bw^n\cdot\nabla u^n
\big) v\, ds+
\int_{\Gamma(t_n)} \! \nabla_{\Gamma}u^n \,\! \cdot\!\nabla_{\Gamma}v \, ds\\ + \rho\int_{\O(\Gamma(t_n))}(\nabla u^{n}\cdot\nabla\phi^n)(\nabla v\cdot\nabla\phi^n)\,dx=-\eps^{-2}\!\! \int_{\Gamma(t_n)}\!\! f(u^{n-1})v\, ds ,
\end{multline}
for all sufficiently smooth test functions $v:\O(\Gamma(t_n))\to\mathbb{R}$. $\rho>0$ is an augmentation parameter for the normal extension condition, {\color{black}and $\bw_T=\bw-(\bw\cdot\bn)\bn$ is the tangential part of $\bw$}.

We need  the integration by parts identity:
\begin{equation} \label{e:intpartc}
\begin{aligned}
 	\int_{\Gamma(t)}(\bw\cdot\nabla u)v\,ds 
 & = \frac12 \int_{\Gamma(t)} (\wt\cdot\nabla_\Gamma u v - \wt\cdot\nabla_\Gamma v u)\,ds
   -\frac12 \int_{\Gamma(t)} (\Div_\Gamma\wt) u v\,ds
\end{aligned}
\end{equation}
for sufficiently smooth $u,v$ such that $\bn\cdot\nabla u=\bn\cdot \nabla v = 0$ (recall that $\bn=\nabla\phi$ on $\Gamma$).

\subsection{Finite element method} \label{s:fulldisc}
To reduce the repeated use of generic but unspecified constants, further in the paper we write $x\lesssim y$ to state that the inequality $x \le cy$ holds for quantities $x$, $y$ with a constant $c$, which is independent of the mesh parameters $h$, $\Delta t$, time
instance $t_n$, and the position of $\Gamma$ and $\Gamma_h$ in the bulk mesh. Similarly we give sense to $x\gtrsim y$. 

Consider a family $\{\T_h\}_{h>0}$ of shape-regular consistent triangulations of the bulk domain $\Omega$, with $\max\limits_{T\in\T_h}\mbox{diam}(T) \le h$. The bulk triangulation supports a standard finite element space of piecewise polynomial
continuous functions of a fixed degree $k\ge1$:
\begin{equation} \label{eq:Vh}
V_h=\{v_h\in C(\Omega)\,:\, v_h|_S\in P_k(S), \forall S\in \mathcal{T}_h\}.
\end{equation}
We next approximate the  sign distance function $\phi$ by a finite element distance function $\phi_h$ {\color{black} of degree $q$, i.e. $\phi_h\in V_h$ for $k=q$}, such that
\begin{equation}\label{phi_h}
\|\phi-\phi_h\|_{L^\infty(\O(\Gamma(t)))}+ h \|\nabla(\phi-\phi_h)\|_{L^\infty(\O(\Gamma(t)))}\lesssim \,h^{q+1},\quad \forall~t\in[0,T],
\end{equation}
where we need to assume $\phi \in C^{q+1}(\O(\Gs))$.
Following \cite{lehrenfeld2018eulerian}, we also assume that $\nabla \phi_h(\bx,t) \neq 0$ in  $\O(\Gamma(t))$, $t\in[0,T]$, and that on every time interval $I_n=[t_{n-1},t_n]$ there holds
\begin{subequations}\label{phi_hb}
\begin{align}\label{phi_hba}
\| \phi_h^{n-1}-\phi_h^n\|_{L^\infty(\Omega)} & \lesssim\,\Delta t \|\bw\cdot\bn\|_{\infty,I_n}, \\
\label{phi_hbb}
\|\nabla \phi_h^{n-1}-\nabla \phi_h^n\|_{L^\infty(\Omega)} & \lesssim\,\Delta t\left( \|\bw\cdot\bn\|_{\infty,I_n} +  \|\nabla (\bw\cdot\bn)\|_{\infty,I_n}\right), \text{ for } n=1,\dots,N,
\end{align}
\end{subequations}
where  $\phi_h^n(\bx) = \phi_h(\bx,t_n),~n=0,\dots,N$, and
$
\norm{v}_{\infty,I_n}
       := \sup\limits_{t\in I_n} \norm{v}_{L^\infty(\Gamma(t))},
$
for $v$ defined on $\Gamma(t)$.

We now introduce the ``discrete'' surfaces $\G{n}{h}$ as the zero level of $\phi_h^n$,
\[
\G{n}{h}:=\{\bx\in\rr^3\,:\,\phi_h^n(\bx)=0\}.
\]
Thanks to \eqref{phi_h} it approximates the original  surface $\Gamma$ in the following sense
\begin{equation}\label{eq:dist}
\operatorname{dist}(\G{n}{h},\Gamma(t_h)) = \max_{x\in\G{n}{h}} |\phi^n(\bx)|
= \max_{x\in\G{n}{h}} |\phi^n(\bx) - \phi_h^n(\bx)| \leq \Vert \phi^n - \phi_h^n \Vert_{L^\infty(\Omega)} \lesssim h^{q+1}.
\end{equation}
For the normal vector to $\G{n}{h}$, $\bn_h^n=\nabla\phi_h^n/|\nabla\phi_h^n|$, and the extended normal vector to $\Gamma(t_n)$, $ \bn^{n} = \nabla\phi^n$, the following consistency bound follows from \eqref{phi_h}:
\begin{equation}\label{eq:normals}
  \vert \bn_h^n(\bx) - \bn^n(\bx) \vert \lesssim | \nabla \phi_h^n(\bx) - \nabla \phi^n(\bx) | \lesssim h^q,\quad \bx \in \G{n}{h}.
\end{equation}

For practical reasons, the finite element method does not look for an extension of the discrete solution to the whole neighborhood $\O(\Gs)$. Instead  it  provides an extension to a  \emph{narrow} band around   $\G{n}{h}$.  For each $n$, the extension band consists   of all tetrahedra on a $\delta_n$ distance from $\G{n}{h}$, for
\begin{equation} \label{e:delta}
  \delta_n := c_\delta \| \bw\cdot\bn\|_{\infty,I_n}~\Delta t
\end{equation}
and $c_\delta\geq 1 $, an $O(1)$ mesh-independent constant. More precisely, we define the {\color{black} mesh-dependent} narrow band as
\begin{equation*}
{\O}_h(\G{n}{h})= {\bigcup}\left\{\overline{S}\,:\,  S\in\mathcal{T}_h\,:| \phi_h^n(\bx) |\le\delta_n  \text{ for some } \bx \in S\right\}.
\end{equation*}
We also need a subdomain of ${\O}_h(\G{n}{h})$  only consisting of tetrahedra intersected by $\G{n}{h}$,
\begin{equation*}
{\O}_\Gamma(\G{n}{h}):=
{\bigcup}\left\{ \overline{S}\in\mathcal{T}_h\,:\, S\cap\G{n}{h}\neq\emptyset\right\}.
\end{equation*}
Since $\mbox{dist}(\G{n}{h},\Gamma(t_n))\lesssim h^{q+1}$, the narrow band width $\delta_n$ and $h$ can be assumed small enough such that
\begin{equation}\label{neigh_cond}
  {\O}_h(\G{n}{h})\subset\O(\Gamma(t_n)). 
\end{equation}
{\color{black} This and \eqref{e:delta}}
implies the restriction on the time step of the form
\begin{equation} \label{eq:conddtkappa}\asslabel
  \Delta t \leq c_0 (c_\delta \norm{\bw\cdot\bn}_{\infty,I_n})^{-1}=O(1),  ~ n = 1,\dots,N,
\end{equation}
with some $c_0$ sufficiently small, but independent of $h$, $\Delta t$ and $n$.
{\color{black} On one time step from $t_{n-1}$ to $t_n$, the surface  may travel up to $\Delta t\|\bw\cdot\bn\|_{\infty,I_n}$ distance in normal directions, which is thus the maximum distance from $\Gamma_h^n$ to $\Gamma_h^{n-1}$.
Therefore}, $c_\delta$  can be taken  sufficiently large, but independent of $h$, such that
\begin{equation}\label{cond1} \asslabel
{\O_\Gamma}(\G{n}{h})\subset{\O}_h(\G{n-1}{h}).
\end{equation}
{\color{black} To see this, one applies \eqref{e:delta} to determine $\delta_{n-1}$, which in turn defines ${\O}_h(\G{n-1}{h})$.}
This condition is the discrete analog of \eqref{ass1} and it is essential for the well-posedness of the finite element formulation below.

Next we define test and trial  finite element spaces of degree $m\ge1$ as restrictions of the time-independent bulk space $V_h$, $k=m$, on all tetrahedra from ${\mathcal{O}}(\G{n}{h})$:
\begin{equation}\label{eq:testtrial}
V_h^n=\{v \in C({\O}_h(\G{n}{h}))\,:\, v\in P_m(S),~ \forall  S\in\mathcal{T}_h,  S\subset {\mathcal{O}}(\G{n}{h})\},\quad m\ge1.
\end{equation}
We further use $V_h^n$ as test and trial spaces in the integral formulation \eqref{e:conttraceFEM1}, where we use identity \eqref{e:intpartc} and replace $\Gamma(t_n)$ by $\G{n}{h}$, $\O(\Gamma(t_n))$ by ${\O}_h(\G{n}{h})$. The resulting FE formulation reads:
{For a given $u_h^0 \in V_h^0$ find  $u_h^n\in V_h^n$, $n=1,\dots,N$, solving}
\begin{multline}\label{e:traceFEM1}
\!\!\int\limits_{\G{n}{h}}\!\left\{\!\!(1+\beta_s\Delta{}t)\frac{u_h^n-u_h^{n-1}}{\Delta t} v_h + \frac12 \left(\wte\cdot\nabla_{\G{}{h}} u^n_h v_h - \wte\cdot\nabla_{\G{}{h}} v_h u_h^n -(\Div_{\G{}{h}}\wte) u_h^n v_h\right)\! \right\}\! ds_h\\ + \int\limits_{\G{n}{h}}\nabla_{\G{}{h}}u^n_h\cdot\nabla_{\G{}{h}}
v_h\, ds_h+\stab \int\limits_{{\O}_h(\G{n}{h})}(\bn_h^n\cdot\nabla u_h^n) (\bn_h^n\cdot\nabla v_h) d\bx =- \eps^{-2}\!\! \int\limits_{\G{n}{h}}\!\! f(u_h^{n-1}) v_h\, ds_h,
\end{multline}
{for all $v_h\in V_h^n$. Here $\bn_h=\nabla\phi_h^n/|\nabla\phi_h^n|$ in ${\O}_h(\G{n}{h})$, $\stab>0$ is a  parameter, $\bw^e(\bx)=\bw(\bp^n(\bx))$ is a lifted data on $\G{n}{h}$ from $\Gamma(t_n)$. The  terms  involving $u^{n-1}$ are well-defined thanks to condition \eqref{cond1}.}
With suitable restrictions on problem parameters the last term on the left-hand side of \eqref{e:traceFEM1} ensures the whole bilinear form is elliptic on $V_h^n$; see \eqref{coer}. Therefore, on each time step we obtain a FE solution defined in  ${\O}_h(\G{n}{h})$ (not just on $\G{n}{h}$ and this can be seen as an implicit extension procedure).
As discussed in many places in the literature, see, e.g. \cite{lehrenfeld2018stabilized}, this term also stabilizes the problem algebraically, i.e. the resulting systems of algebraic equations are well-conditioned independent on how the surface $\Gamma_h$ cuts through the ambient triangulation.

We finally note that an accurate integration  over $\G{n}{h}$ may be not feasible using standard quadrature rules for higher than second order surface representation, i.e. for  $q>1$. More sophisticated numerical integration techniques should be applied as discussed in the literature \cite{fries2015,lehrenfeld2015cmame,grande2018analysis,muller2013highly,saye2015hoquad,sudhakar2013quadrature}.

\section{Analysis of the finite element method} \label{s:Analysis}
In this section we address stability and error analysis of the finite element formulation \eqref{e:traceFEM1}.
For a proper control of the geometric error, the analysis  requires the following mild restriction on the mesh step,
\begin{equation}\label{cond4}\asslabel
h^{2q} \lesssim \Delta t.
\end{equation}
We recall that $q\ge 1$ is the degree of geometry approximation from \eqref{phi_h}.

We shall need the following two Lemmas from \cite{lehrenfeld2018stabilized}. The result of the first lemma allows the control of the $L^2$ norm of $v_h  \in V_h^n$ in the narrow band by its $L^2$ norm on $\Gamma_h$ and a term similar to the normal volume stabilization in \eqref{e:traceFEM1}.  While the second lemma provides control over the $L^2$ norm of the extension of a FE function on $\G{n}{h}$ by its values on $\G{n-1}{h}$. That lemma is essential for applying a Gronwall type argument later.
\begin{lemma}\label{lemcrucial}
  Assume  conditions \eqref{e:delta} and \eqref{eq:conddtkappa} are satisfied, then for any $v_h \in V_h^n$ it holds
\begin{equation}
    \label{fund1}
\|v_h \|_{{\O}_h(\G{n}{h})}^2  \lesssim (\delta_n + h) \|v_h \|_{\G{n}{h}}^2 +(\delta_n + h)^2 \|\bn_h^n \cdot \nabla v_h \|_{{\O}_h(\G{n}{h})}^2.
\end{equation}
\end{lemma}

\begin{lemma}\label{lem2} In addition to  \eqref{e:delta} and \eqref{eq:conddtkappa} assume \eqref{cond4} is satisfied.
{\color{black}Assume $\V_h$ is a subset of  $H^1\left({\O}_h(\G{n-1}{h})\right)$ that supports the following inequalities:
\begin{equation}\label{Lem_ass}
\|\nabla v\|_S\le C h^{-1}\|v\|_S,\quad
\|\nabla v\|_D\le C |D||S|^{-1}\|\nabla v\|_S,\quad
\|\bn_h^{n-1} \cdot\nabla v\|_D\le C |D||S|^{-1}\|\bn_h^{n-1} \cdot\nabla v\|_S,
\end{equation}
for all $v\in\V_h$, $S\in\T_h$, $S\subset{\O}_h(\G{n-1}{h})$, where $D$ is a subdomain in $S$, and $C$ depends only on the shape-regularity of $S$.}
Then for any $v \in \V_h$ it holds
 \begin{equation} \label{fund2}
   \|v\|_{\G{n}{h}}^2
   \le
    (1+c_1\Delta t)
    \|v\|_{\G{n-1}{h}}^2 + c_2 \delta_{n-1} (\delta_{n-1}+h)^{-1} \|\bn_h^{n-1} \cdot \nabla v\|_{{\O}_h(\G{n-1}{h})}^2,
\end{equation}
for some $c_1$ and $c_2$ independent of $h$, $\Delta t$ and $n$.
\end{lemma}
\begin{proof}
{\color{black} For $\V_h=V_h^n$ the result is found as Lemma~9 in  \cite{lehrenfeld2018stabilized}. The examination of the proof reveals that
inequalities in \eqref{Lem_ass} are the only assumptions required to extend the result from $V_k^n$ to a more general subset of $H^1\left({\O}_h(\G{n-1}{h})\right)$.
}
\end{proof}


\subsection{Stability analysis} \label{s_stab}
In addition to \eqref{eq:conddtkappa}, we need another $O(1)$ restriction on the time step:
\begin{equation}\label{cond2}\asslabel
  \Delta t \le (4\xi_h)^{-1} \text{ with } \xi_h :=  \frac12 \max_{n=0,..,N} \Vert \Div_{\G{}{h}} \wte \Vert_{{\infty},\G{n}{h}}.
\end{equation}
From the definition of $\xi_h$, smoothness of $\bw$,  and geometry approximation condition \eqref{phi_h}, it follows that
\begin{equation}\label{ass2}
\xi_h\lesssim 1. 
\end{equation}
The normal volume stabilization parameter $\stab$ in \eqref{e:traceFEM1} should be chosen  to satisfy:
\begin{equation}\label{cond3}\asslabel
  \stab \ge C_\rho (\delta_n+h)^{-1}
\end{equation}
with some sufficiently large, but independent of $\Delta t$ and $h$, constant $C_\rho>0$ .
Recalling that $\delta_n\lesssim \Delta t$ (see \eqref{e:delta})
we see that \eqref{cond3} leads an $O((\Delta t +h)^{-1})$ lower bound on $\stab$.
For the stabilization parameter $\beta_s$ we assume
\begin{align}\label{stab_beta}
\beta_s\ge 2\xi_h+\eps^{-2}L+1.
\end{align}
{\color{black} It is noted already in \cite{Shen_Yang2010} that the stabilization term with $\beta_s\simeq \eps^{-2}$ introduces the consistency error of the same order as the explicit treatment of $f$.}
With the help of \eqref{e:delta} and \eqref{cond3} we obtain the inequality
\[c_2 \delta_{n-1} (\delta_{n-1}+h)^{-1}\le {
c_2 c_\delta \| \bw\cdot\bn\|_{\infty,I_{n-1}}}\Delta t C_\rho^{-1}\stabold.\]
Using this, estimate \eqref{fund2} for $C_\rho$ large enough, {\color{black} i.e. such that $C_\rho\ge c_2 c_\delta \| \bw\cdot\bn\|_{\infty}$},  we get
 \begin{equation} \label{fund2simple}
   \|v_h\|_{\G{n}{h}}^2
   \le
    (1+c_1 \Delta t)
    \|v_h\|_{\G{n-1}{h}}^2 + \stabold \Delta t \|\bn_h^{n-1} \cdot \nabla v_h\|_{{\O}_h(\G{n-1}{h})}^2 \quad \forall\,v_h \in V_h^{n-1}.
\end{equation}

For the sake of convenience,  we define the bilinear form on $ H^1({\O}_h(\G{n}{h}))\times  H^1({\O}_h(\G{n}{h}))$:
\begin{equation} \label{e:def:an}
  \begin{split}
    a_n(u,v) := &
   \frac12  \int_{\G{n}{h}} \left( (\wte\cdot\nabla_{\G{}{h}} u) v - (\wte\cdot\nabla_{\G{}{h}} v) u  -(\Div_{\G{}{h}}\wte) u v\right) \, ds \\
    & +\int_{\G{n}{h}}(\nabla_{\G{}{h}}u)\cdot(\nabla_{\G{}{h}}v)\, ds+\stab\int_{{\O}_h(\G{n}{h})}(\bn_h^n\cdot\nabla u)(\bn_h^n\cdot\nabla v) d\,\bx.
  \end{split}
\end{equation}
Because of obvious cancellations, $a_n(v_h,v_h)$ satisfy the lower bound:
\begin{equation}\label{lower}
a_n(v_h,v_h)\ge
\|\nabla_{\G{}{h}}v_h\|^2_{\G{n}{h}} - \xi_h \|v_h\|^2_{\G{n}{h}} +\stab \|\bn_h^n \cdot \nabla v_h\|_{{\O}_h(\G{n}{h})}^2,\quad \forall\,v_h\in V^n_h.
\end{equation}
The low bound  \eqref{lower} and condition \eqref{cond2} imply that the bilinear form on the left-hand side of \eqref{e:traceFEM1} is positive definite,
\begin{equation}\label{coer}
\begin{split}
\int_{\G{n}{h}}\frac{1+\beta_s \Delta t}{\Delta t}v_h^2\,ds+a_n(v_h,v_h)&\ge
\frac{1+2\beta_s \Delta t}{2\Delta t}\|v_h\|^2_{\G{n}{h}}+\|\nabla_{\G{}{h}}v_h\|^2_{\G{n}{h}}+\stab\|\bn_h^n \cdot \nabla v_h\|_{{\O}_h(\G{n}{h})}^2.
\end{split}
\end{equation}
From \eqref{fund1} it follows that the square root of the right-hand side in \eqref{coer} defines a norm on $V_h^{n}$.
Hence, due to the Lax-Milgram lemma, \textit{the problem in each time step of \eqref{e:traceFEM1} is well-posed}.

\medskip
We next derive an \textit{a priori} estimate for the finite element solution to \eqref{e:traceFEM1}.
\begin{theorem}\label{Th1}
  Assume conditions \eqref{e:delta},
  \eqref{eq:conddtkappa},
  \eqref{cond4},
  \eqref{cond2},
  \eqref{cond3}, and \eqref{stab_beta}, then the solution of \eqref{e:traceFEM1} satisfies the following stability estimate:
  \begin{align}\label{FE_stab1}
   \|u_h^n \|^2_{\G{n}{h}}
   +\Delta t\eps^{-2}\!\! \int_{\G{n}{h}}\!\! F(u_h^n) ds_h
   +{\Delta t}\sum_{k=1}^{n}\left(\|\nabla_{\G{}{h}}u_h^k\|^2_{\G{k}{h}} +\Delta{}t\rho_k\|\bn_h^k\cdot\nabla{}u_h^k\|_{\O(\G{k}{h})}^2\right)
   \leq c_0,\quad {\small k=1,\dots,N},
\end{align}
where $c_0$ is independent of $\Delta t$, $h$, $n$ and position of $\Gamma_h$ in the mesh, but depends on $u_0$, $\eps$, and $M$.
\end{theorem}
\begin{proof}
  We test \eqref{e:traceFEM1} with $v_h=u_h^n$ to arrive at the equality
\[
\frac{1+\beta_s\Delta{}t}{2\Delta t}(\|u_h^n\|^2_{\G{n}{h}}+\|u_h^n-u_h^{n-1}\|^2_{\G{n}{h}})+a_n(u_h^n,u_h^n)=\frac{1+\beta_s\Delta{}t}{2\Delta t}\|u_h^{n-1}\|^2_{\G{n}{h}}- \eps^{-2}\!\! \int_{\G{n}{h}}\!\! f(u_h^{n-1}) u_h^n\, ds_h.
\]
From the Taylor expansion  we get
\begin{align}
\label{taylor_nonlinear}
- \!\! \int_{\G{n}{h}}\!\! f(u_h^{n-1}) u_h^n\, ds_h= \!\! \int_{\G{n}{h}}\!\! F(u_h^{n-1})-F(u_h^{n})\, ds_h- \!\! \int_{\G{n}{h}}\!\! f(u_h^{n-1}) u_h^{n-1}\, ds_h + \!\! \int_{\G{n}{h}}\!\!\frac{f'\!(
c)}{2}(u_h^n-u_h^{n-1})^2 ds_h
\end{align}
with some $c \in C(\G{n}{h})$.

We  bound $a_n(u_h^n,u_h^n)$ from below through \eqref{lower} and further use \eqref{growth}, \eqref{variation}, \eqref{fund2simple}, and \eqref{taylor_nonlinear} to arrive at
\begin{align}\label{est1}
&(1+\Delta t(\beta_s- 2\xi_h) )\|u_h^n \|^2_{\G{n}{h}}
+{2\Delta t}\|\nabla_{\G{}{h}}u_h^n\|^2_{\G{n}{h}}
+2{\Delta t}\stab\|\bn_h^n\cdot\nabla{}u_h^n\|_{{\O}_h(\G{n}{h})}^2
+2\Delta t\eps^{-2}\!\! \int_{\G{n}{h}}\!\! F(u_h^n) ds_h \nonumber \\
 &\le (1+\Delta{}t(\beta_s+2\eps^{-2}))\left[(1 + c_1 \Delta t ) \|u_h^{n-1}\|^2_{\G{n-1}{h}} +\Delta t \stabold  \|\bn_h^{n-1} \cdot \nabla u_h^{n-1} \|_{{\O}_h(\G{n-1}{h})}^2\right]\nonumber\\
 &\quad +2\Delta t\eps^{-2}\!\! \int_{\G{n}{h}}\!\! F(u_h^{n-1}) ds_h
 - (1+\Delta t(\beta_s-\eps^{-2}L))\|u_h^n-u_h^{n-1}\|^2_{\G{n}{h}}.
\end{align}
Using \eqref{stab_beta} simplifies the above estimate to
\begin{align}\label{est12}
&\|u_h^n \|^2_{\G{n}{h}}
+2{\Delta t}\|\nabla_{\G{}{h}}u_h^n\|^2_{\G{n}{h}}
+2{\Delta t}\stab\|\bn_h^n\cdot\nabla{}u_h^n\|_{{\O}_h(\G{n}{h})}^2
+2\Delta t\eps^{-2}\!\! \int_{\G{n}{h}}\!\! F(u_h^n) ds_h \nonumber \\
 &\le (1 + c \Delta t )\left[ \|u_h^{n-1}\|^2_{\G{n-1}{h}} +\Delta t \stabold  \|\bn_h^{n-1} \cdot \nabla u_h^{n-1} \|_{{\O}_h(\G{n-1}{h})}^2\right]
 +2\Delta t\eps^{-2}\!\! \int_{\G{n}{h}}\!\! F(u_h^{n-1}) ds_h,
\end{align}
where the constant $c$ is independent of $h$, $\Delta t$ and $n$.

We further estimate the $F$-term on the right-hand side employing Lemma~\ref{lem2}
and the elementary inequality $|\left(\sqrt{F(x)}\right)_x|\le C =2M$ for almost all $x\in\mathbb{R}$:
\begin{align*}
	\int_{\G{n}{h}}\!\! F(u_h^{n-1}) ds_h&=\|\sqrt{F(u_h^{n-1})}\|_{\G{n}{h}}^2\\
&\leq (1+c_1\Delta t)
	\|\sqrt{F(u_h^{n-1})}\|_{\G{n-1}{h}}^2 + c_2 \delta_{n-1} (\delta_{n-1}+h)^{-1} \|\bn_h^{n-1} \cdot \nabla \sqrt{F(u_h^{n-1})}\|_{{\O}_h(\G{n-1}{h})}^2\nonumber\\
		&\leq(1+c_1\Delta t)
	\int_{\G{n-1}{h}}\!\! F(u_h^{n-1}) ds_h + c_3\rho_{n-1} \Delta{}t \|  (\bn_h^{n-1} \cdot \nabla u_h^{n-1})\|^2_{{\O}_h(\G{n-1}{h})}
 	\end{align*}
with $c_2$ and $c_3$ independent of problem parameters.
Substituting  this into \eqref{est12} we obtain the estimate
\begin{align*}
&\|u_h^n \|^2_{\G{n}{h}}
+{2\Delta t}\|\nabla_{\G{}{h}}u_h^n\|^2_{\G{n}{h}}
+2\Delta{}t\stab\|\bn_h^n\cdot\nabla{}u_h^n\|_{{\O}_h(\G{n}{h})}^2
+2\Delta t\eps^{-2}\!\! \int_{\G{n}{h}}\!\! F(u_h^n) ds_h
 \nonumber \\
 &\le \left( 1+c\,\Delta{}t\right) \left( \|u_h^{n-1}\|^2_{\G{n-1}{h}}
+  2\Delta{}t \stabold\|\bn_h^{n-1} \cdot \nabla u_h^{n-1} \|_{{\O}_h(\G{n-1}{h})}^2
+2 \Delta t\eps^{-2}\!\! \int_{\G{n-1}{h}}\!\! F(u_h^{n-1}) ds_h\right),\nonumber
\end{align*}
with some $c$ independent of $h$, $\Delta t$ and $n$. Applying discrete Gronwall inequality proves the theorem.
\end{proof}

We now proceed with a consistency estimate and further combine it and interpolation bounds with the above stability analysis to arrive at an error estimate in the energy norm. Thanks to  the hybrid (FD in time -- FE in space) structure of the discretization method, geometric  and interpolation error estimates will be computed on each time step for a `steady' surface $\G{n}{h}$. This allows re-using  the consistency and error bounds from \cite{reusken2015analysis,olshanskii2016trace}.

\subsection{Consistency estimate} \label{s:consistency}
For parameter  $\stab$ we earlier required the lower bound \eqref{cond3}.  For optimal order consistency we now assume a similar \emph{upper} bound:
\begin{equation}\label{cond5}\asslabel
 \stab \lesssim (h+\delta_n)^{-1}.
\end{equation}
Substituting in \eqref{e:traceFEM1} $u^n=u(t_n)$ for the smooth solution $u(t)$ of \eqref{transport1} we obtain
\begin{equation}\label{e:consist}
\int_{\G{n}{h}}(1+\beta_s\Delta{}t)\left(\frac{u^n-u^{n-1}}{\Delta t} \right) v_h\,ds+ a_n(u^n,v_h) +\eps^{-2}\!\!\int_{\G{n}{h}}f(u^{n-1})v_h\,ds = \consist(v_h),\quad \forall~v_h\in V^n_h,
\end{equation}
{with  $\consist(v_h)$ collecting consistency terms due to geometric errors, time derivative approximation} and nonlinear term, i.e.
\begin{equation*}
\begin{split}
\consist(v_h)&=
\underset{I_1}{\underbrace{\int_{\G{n}{h}}(1+\beta_s\Delta{}t)\left(\frac{u^n-u^{n-1}}{\Delta t}\right) v_h\, ds_h-\int_{\Gamma(t_n)}u_t(t_n) v_h^\ell\, ds}}\\
&\quad+
\underset{I_2}{\underbrace{\stab\int_{{\O}_h(\G{n}{h})}((\bn_h^n-\bn^n)\cdot\nabla u^n)(\bn_h^n\cdot\nabla v_h) d\bx}}\\
&\quad
+\underset{I_{3,a}}{\underbrace{
    \frac12 \int_{\G{n}{h}} \wte\cdot\nabla_{\G{}{h}} u^n v_h - \wte\cdot\nabla_{\G{}{h}} v_h u^n\, ds_h
    - \frac12 \int_{\Gamma(t_n)} \bw \cdot \nabla u^n v_h^\ell - \bw \cdot \nabla v_h^\ell u^n \, ds}}\\
&\quad
+\underset{I_{3,b}}{\underbrace{\frac12\int_{\Gamma(t_n)}\ \Div_{\Gamma} (\wt) u^n v_h^\ell\, ds-\frac12\int_{\G{n}{h}} \Div_{\G{}{h}} (\wte) u^n v_h\, ds_h}}\\
&\quad
+\underset{I_4}{\underbrace{\int_{\G{n}{h}}\nabla_{\G{}{h}}u^n\cdot\nabla_{\G{}{h}} v_h\, ds_h - \int_{\Gamma(t_n)}\nabla_{\Gamma}u^n\cdot\nabla_{\Gamma} v_h^\ell\, ds}}\\
&\quad+\underset{I_5}{\underbrace{\eps^{-2}\int_{\G{n}{h}}f(u^{n-1}) v_h\, ds_h - \eps^{-2}\int_{\Gamma(t_n)}f(u^n) v_h^\ell\, ds}},
\end{split}
\end{equation*}
where $v_h^\ell$ is the lifting of $v_h$ to $\Gamma(t_n)$ as defined in section~\ref{s:prelim}.
An  estimate for consistency terms is given in the following lemma.

{
\begin{lemma}\label{l_consist} Let $u\in W^{2,\infty}(\Gs)$.
The consistency error satisfies the bound
\begin{equation}\label{est:consist}
|\consist(v_h)|\lesssim (\Delta t+h^q) \norm{u}_{W^{2,\infty}(\Gs)}\left(\|v_h\|_{\G{n}{h}}+\|\nabla_{\Gamma} v_h\|_{\G{n}{h}} +\stab^{\frac12}\|(\bn_h^n\cdot\nabla v_h)\|_{{\O}_h(\G{n}{h})}\right).
\end{equation}
\end{lemma}}
\begin{proof}
The required estimate for $I_1,\dots,I_4$ is found in \cite{lehrenfeld2018stabilized}. The last term $I_5$ gets estimated as
\begin{align*}
\eps^{2}|I_5|&=\left|\int_{\G{n}{h}}\left(f(u^{n-1})-f(u^n)\right) v_h\, ds_h+\int_{\G{n}{h}}f(u^{n}) v_h\, ds_h  - \int_{\Gamma(t_n)}f(u^n) v_h^\ell\, ds\right|\\
&\leq L \,\int_{\G{n}{h}}\left|\left(u^{n-1}-u^n\right) v_h\right|\, ds_h
 +\left|\int_{\G{n}{h}}f(u^{n})(1 - \mu_h)  v_h\, ds_h\right|\\
 &\leq L\Delta{}t \|u_t\|_{L^\infty(\O(\Gs))} \|v_h\|_{\G{n}{h}}+L\|u\|_{L^\infty(\O(\Gs))}h^{q+1}  \|v_h\|_{\G{n}{h}}\lesssim{}(\Delta{}t+h^{q+1})\|u\|_{W^{2,\infty}(\Gs)}\|v_h\|_{\G{n}{h}}.
\end{align*}
Here we have used $\mu_h d s_h(\bx)=ds(p(\bx))$, $\bx\in\Gamma_h^n$, with $\|1-\mu_h\|_{\infty,\Gamma_h^n}\leq h^{q+1}$ (cf. \cite{reusken2015analysis}).
\end{proof}

\subsection{Error estimate in the energy norm} \label{sec:aprioriest}
{Denote the error function by $\err^n=u^n-u^n_h$, $\err^n\in H^1({\O}_h(\G{n}{h}))$. From \eqref{e:traceFEM1} and \eqref{e:consist} we get the error equation}, for $v_h\in V^n_h$:
\begin{equation}\label{e:err}
\int_{\G{n}{h}}(1+\beta_s\Delta{}t)\left(\frac{\err^n-\err^{n-1}}{\Delta t} \right) v_h\,ds+ a_n(\err^n,v_h) +\eps^{-2}\int_{\G{n}{h}} (f(u^{n-1})-f(u^{n-1}_h)) v_h\,ds = \consist(v_h).
\end{equation}
{We  assume  $u^n$ sufficiently smooth  in  ${\O}_h(\G{n}{h})$ so that the nodal interpolant $u_I^n\in V_h^n$ is well-defined. We split $\err^n$ into finite element and approximation parts,
\[
\err^n=\underset{\mbox{$e^n$}}{\underbrace{(u^n-u^n_I)}}+\underset{\mbox{$e^n_h$}}{\underbrace{(u^n_I-u^n_h)}}.
\]
From \eqref{e:err} we get
\begin{equation}\label{e:err1}
\int_{\G{n}{h}}(1+\beta_s\Delta{}t)\left(\frac{e^n_h-e^{n-1}_h}{\Delta t} \right) v_h\,ds+ a_n(e_h^n,v_h) +\eps^{-2}\int_{\G{n}{h}} (f(u^{n-1}_I)-f(u^{n-1}_h)) v_h\,ds= \interpol(v_h)+\consist(v_h),
\end{equation}
for any $v_h\in V^n_h$, and
\[
\interpol(v_h)=-(1+\beta_s\Delta{}t)\int_{\G{n}{h}}\left(\frac{e^n-e^{n-1}}{\Delta t} \right) v_h\,ds_h- a_n(e^n,v_h)-\eps^{-2}\int_{\G{n}{h}} (f(u^{n-1})-f(u^{n-1}_I)) v_h\,ds.
\]
An estimate for these interpolation terms is given  in the following lemma.
Further we assume $\Gs$ sufficiently smooth to support functions from $W^{m+1,\infty}(\Gs)$.

{
\begin{lemma}\label{l_interp} Assume $u\in W^{m+1,\infty}(\Gs)$,
then it holds
\begin{equation}\label{est_inter}
  |\interpol(v_h)|\lesssim h^m \, \norm{u}_{W^{m+1,\infty}} \, (\|v_h\|_{\G{n}{h}}+\|\nabla_{\Gamma_h} v_h\|_{\G{n}{h}}).
\end{equation}
\end{lemma}}
\begin{proof}
We only need to estimate the third term of $\interpol(v_h)$.  The required bound for other terms is given in \cite{lehrenfeld2018stabilized}.
 We make use of  the  following local trace inequality, cf.~\cite{Hansbo02,reusken2015analysis,guzman2018inf}:
\begin{equation}\label{trace}
\|v\|_{S\cap\G{n}{h}}\le c(h^{-\frac12}\|v\|_{S}+h^{\frac12}\|\nabla v\|_{S}),\quad ~~v\in H^1(S),~~S\in \mathcal{T}_h^{\Gamma},
\end{equation}
with some $c$ independent of $v$, $T$, $h$, and position of $\G{n}{h}$ in $S$.
We need interpolation properties of polynomials and their traces \cite{gross2015trace,reusken2015analysis}:
\begin{equation}\label{eq:interp}
\|v^e-v_I\|_{\O_\Gamma(\G{n}{h})}+h\|\nabla(v^e-v_I)\|_{\O_\Gamma(\G{n}{h})}\lesssim h^{m+1}\|v^e\|_{H^{m+1}(\O_\Gamma(\Gamma(t_n)))}\quad \text{for}~v\in H^{m+1}(\Gamma(t_n)).
\end{equation}
With the help of \eqref{growth}, \eqref{trace}, and \eqref{eq:interp} we estimate
\begin{align*}
	&\left|\int_{\G{n}{h}} (f(u^{n-1})-f(u^{n-1}_I)) v_h\,ds\right|\leq   L\int_{\G{n}{h}} \left|(u^{n-1}-u^{n-1}_I) v_h\right|ds\leq L\left\|{e^{n-1}}\right\|_{\G{n}{h}}\|v_h\|_{\G{n}{h}}\\
	&\lesssim
	h^{-\frac12}\left( \|e^{n-1}\|_{\O_\Gamma(\G{n}{h})}+h  \|\nabla e^{n-1}\|_{\O_\Gamma(\G{n}{h})}\right)\|v_h\|_{\G{n}{h}}\lesssim{}h^{-\frac12}\left( h^{m+1}\| u\|_{H^{m+1}(\O_\Gamma(\G{{n-1}}{}))}\right)\|v_h\|_{\G{n}{h}}\\
	&\lesssim{}h^{-\frac12}\left( h^{m+\frac32}\| u\|_{W^{m+1,\infty}(\Gs)}\right)\|v_h\|_{\G{n}{h}}\lesssim{}h^{m+1}\| u\|_{W^{m+1,\infty}(\Gs)}\|v_h\|_{\G{n}{h}}.
\end{align*}

\end{proof}

Now we are prepared to prove the main result of the paper. Let $u_h^0=u_I^0\in V_h^0$ be a nodal interpolant to
$u^0\in \O(\Gamma_h^0)$.

\begin{theorem}\label{Th2}
{  Assume \eqref{phi_h}--\eqref{phi_hbb}, \eqref{e:delta}, \eqref{eq:conddtkappa},
  \eqref{cond4},
  \eqref{cond2},
  \eqref{cond3}, \eqref{stab_beta}, and
  \eqref{cond5}.  Solution  $u$  to \eqref{transport} is such that $u\in W^{m+1,\infty}(\Gs)$.
  For  $u_h^n$, $n=1,\dots,N$, the finite element solution of \eqref{e:traceFEM1}, $u^n=u(t_n)$,
  and the error function $\err^n = u_h^n - u^n$  the following estimate holds:}
  \begin{equation}\label{FE_est1}
    \|\err^n\|^2_{\G{n}{h}}+{\Delta t}\sum_{k=1}^{n} \|\nabla_{\G{}{h}} \err^k \|^2_{\G{k}{h}}
    \lesssim \exp(c\, t_n) \norm{u}_{W^{m+1,\infty}(\Gs)}^2 (\Delta t^2+h^{2\min\{m,q\}}),
  \end{equation}
  with $c$ independent of $h$, $\Delta t$, $n$ and  the position of the surface in the background mesh.
\end{theorem}
\begin{proof}
Letting $v_h=2\Delta t e^n_h$ in \eqref{e:err1} gives
\begin{multline*}
(1+\beta_s\Delta{}t)\left(\|e_h^n\|^2_{\G{n}{h}}- \|e_h^{n-1}\|^2_{\G{n}{h}} +\|e_h^n-e_h^{n-1}\|^2_{\G{n}{h}}\right)+{2\Delta t} a_n(e_h^n,e_h^n)\\+2\eps^{-2}\Delta{}t\int_{\G{n}{h}} (f(u^{n-1}_I)-f(u^{n-1}_h))  e^n_h\,ds=2\Delta t(\interpol(e_h^n)+\consist(e_h^n)).
\end{multline*}
The nonlinear term is estimated using \eqref{variation}:
\[
2\int_{\G{n}{h}} (f(u^{n-1}_I)-f(u^{n-1}_h))  e^n_h\,ds\le{}2\int_{\G{n}{h}}L |e^{n-1}_h|  |e^n_h|\,ds
\leq{}L(\|e_h^{n-1}\|^2_{\G{n}{h}}+\|e_h^{n}\|^2_{\G{n}{h}}).
\]
Dropping out the third term, using the lower bound \eqref{lower} for $a_n$ and applying \eqref{fund2simple} to bound $ \|e_h^{n-1}\|^2_{\G{n}{h}}$  yields
\begin{multline}\label{est2}
(1+(\beta_s-{2\xi_h}-\eps^{-2}L)\Delta{}t)\|e_h^n\|^2_{\G{n}{h}}+{2\Delta t}\|\nabla_{\G{}{h}}e_h^n\|^2_{\G{n}{h}}
+2{\Delta t}{\stab}\|\bn_h^n \cdot \nabla e_h^n\|_{{\O}_h(\G{n}{h})}^2- 2\Delta t(\interpol(e_h^n)+\consist(e_h^n))\\
\le(1+(\beta_s+\eps^{-2}L)\Delta{}t)(1+c_1 \Delta t)
\|e_h^{n-1}\|_{\G{n-1}{h}}^2 + \stabold \Delta t \|\bn_h^{n-1} \cdot \nabla e_h^{n-1}\|_{{\O}_h(\G{n-1}{h})}^2.
\end{multline}
To estimate the interpolation and consistency terms, we apply Young's inequality to the right-hand sides of \eqref{est:consist} and \eqref{est_inter} yielding
\begin{align*}
2 \Delta t\consist(e_h^n)
&\le c\,\Delta t (\Delta t^2+h^{2q}) \norm{u}_{W^{2,\infty}(\Gs)}^2
+\frac{\Delta t}{2}\left(\|e_h^n\|_{\G{n}{h}}^2+\|\nabla_{\Gamma_h} e_h^n\|_{\G{n}{h}}^2 +\stab\|(\bn_h^n\cdot\nabla e_h^n)\|_{{\O}_h(\G{n}{h})}^2\right),\\
2 \Delta t\interpol(e_h^n)
& \le c\,\Delta t ~ h^{2m} \norm{u}_{W^{m+1,\infty}(\Gs)}^2
+\frac{\Delta t}{2}\left(\|e_h^n\|_{\G{n}{h}}^2+\|\nabla_{\Gamma_h} e_h^n\|_{\G{n}{h}}^2\right),
\end{align*}
with a constant $c$ independent of $h$, $\Delta t$, $n$ and of the position of the surface in the background mesh.
By substituting above estimates in \eqref{est2} we get
\begin{multline*}
(1+(\beta_s-{2\xi_h}-\eps^{-2}L-1)\Delta{}t)\|e_h^n\|^2_{\G{n}{h}}+{\Delta t}\|\nabla_{\G{}{h}}e_h^n\|^2_{\G{n}{h}}
+{\stab}{\Delta t}\|\bn_h^n \cdot \nabla e_h^n\|_{{\O}_h(\G{n}{h})}^2\\
\le(1+(\beta_s+\eps^{-2}L)\Delta{}t) (1+c_1 \Delta t)
\|e_h^{n-1}\|_{\G{n-1}{h}}^2 + \stabold \Delta t \|\bn_h^{n-1} \cdot \nabla e_h^{n-1}\|_{{\O}_h(\G{n-1}{h})}^2\\
 +c\Delta{}t\,\norm{u}_{W^{m+1,\infty}(\Gs)}^2 (\Delta t^2+h^{2q}+h^{2m}).
\end{multline*}
Using lower bound \eqref{stab_beta} for $\beta_s$ leads to
\begin{multline*}
\|e_h^n\|^2_{\G{n}{h}}+{\Delta t}\|\nabla_{\G{}{h}}e_h^n\|^2_{\G{n}{h}}
+{\Delta t}{\stab}\|\bn_h^n \cdot \nabla e_h^n\|_{{\O}_h(\G{n}{h})}^2
\le{}(1+c\Delta{}t)
\|e_h^{n-1}\|_{\G{n-1}{h}}^2 \\+ \Delta t\stabold  \|\bn_h^{n-1} \cdot \nabla e_h^{n-1}\|_{{\O}_h(\G{n-1}{h})}^2
+ c\Delta{}t\,\norm{u}_{W^{m+1,\infty}(\Gs)}^2 (\Delta t^2+h^{2q}+h^{2m}),
\end{multline*}
with a constant {$c$} independent of $h$, $\Delta t$, $n$ and of the position of the surface in the background mesh.
Applying the discrete Gronwall inequality proves the theorem.\\
\end{proof}

\section{Numerical experiments}\label{s:Numerics}
In this section, we present results of several numerical experiments, which illustrate the finite element method performance and analysis.
{\color{black} In examples we consider rigid surface motions or small oscillations of a surface, which is consistent with our assumption of small or area-preserving deformations.} All experiments are done using the finite element package DROPS \cite{DROPS}.
To build computation mesh, we use the combination of uniform subdivision into cubes with side length $h$ and the Kuhn subdivision of each cube into 6 tetrahedra. This provides us with a shape regular bulk triangulation $\T_h$.
The temporal grid is uniform in all experiments, $t_n=n\Delta t$ with $\Delta t=\frac{T}{N}$.
We use piecewise linear bulk finite element space $V_h$ (e.g., \eqref{eq:Vh} with $m=1$) for both finite element level set function and for the definition of test and trial spaces in \eqref{eq:testtrial}. This leads to geometry approximation \eqref{phi_h} with $q=1$,

{\bf Example 1.} In the first example, we consider the Allen--Cahn equation on a  sphere moving with constant velocity
$\mathbf{w}=(2,0,0)^T$.
The corresponding level set function is given by
\begin{equation}
({x}-{x}_0(t))^2+({y}-{y}_0(t))^2+({z}-{z}_0(t))^2=1,
\end{equation}
with the center $\mathbf{x}_0(t)=(x_0,y_0,z_0)^T=\mathbf{w}t$.
We consider the Allen--Cahn equation with nonzero right hand side term:
\begin{equation}\label{e:examp1}
\dot{u}+(\mathrm{div}_{\Gamma} \mathbf{w}) u-\Delta_{\Gamma} u-\eps^{-2} f(u)= g(\bx) \quad\hbox{on } \Gamma(t)
\end{equation}
such that  solution is known explicitly:
$$
u=\frac{1}{2}(1-0.8 e^{-40t})\left(\sqrt{\frac{3}{\pi}}(y-y_0)+1\right).
$$

We set $\eps=0.1$, $T=0.1$. {\color{black} According to \eqref{stab_beta} we need $\beta_s$ of order $\eps^{-2}$, so we set $\beta_s=0.2\eps^{-2}$ in all further examples. We observed that in practice the stabilization term cannot be completely omitted without server restrictions on the time step. We do not study however  the optimal choice of parameter $\beta_s$.} The computational domain is $\Omega=[-2,2]^3$; it contains $\G{}{}(t)$ (and $\G{}{h}(t)$) at all times $t \in [0,T]$.
The error is measured in the $L^2(0,T;H^1(\G{}{h}(t)))$ and $L^\infty(0,T;L^2(\G{}{h}(t)))$ surface norms. The former is computed  with the help of the composite trapezoidal quadrature rule in time and the latter is approximate by $\max_{n=1,..,N} \norm{\cdot}_{L^2(\G{n}{h}(t))}$.
Table~\ref{tab:table1} shows the results of experiment.
To study the convergence rates, we apply successive refinements in space and in time.
The ``experimental orders of convergence''(\eoc{}) in space and time are then defined as \eoc{} $=  \log_2(e_b/e_a)$, where $e_a$ and $e_b$ are corresponding error norms. In particular, \eoc{x} stands for the convergence order in space, when  time is fixed. Likewise,  \eoc{tt} shows convergence order in time per two refining steps; and \eoc{xtt} indicates the order for the simultaneous space and time refinement.
From Table~\ref{tab:table1}, we can see that in  $L^2(0,T;H^1(\G{}{h}(t))$ norm the error converges with the first order both in space and time (this agrees with our analysis), while the $L^\infty(0,T;L^2(\G{}{h}(t))$ norm of the error reduces approximately four times if the mesh size is reduced two times and the time step is reduced four times.
The observed rates are optimal for our choice of the finite element space and time-stepping scheme.

\begin{table}
\caption{$L^2(H^1)$- and $L^{\infty}(L^2)$-norm error in Experiment~1 with backward Euler.}\label{tab:table1}
\vspace{-0.2cm}
\begin{center}
  \footnotesize
  \begin{tabular}{@{~}l@{~~}r@{~~}r@{~~}r@{~~}r@{~}r@{~}}
    \toprule
    \multicolumn{5}{c}{$L^2(H^1)$-norm of the error} \\
    \midrule
    & $h=1/2$
    &$h=1/4$
    &  $h=1/8$
    & $h=1/16$ & \eoc{tt} \\
    \midrule
    {$\Delta t=T/64$}
    &\numbf{0.533131}	
    &\num{0.515529}	
    &\num{0.51057}	
    &\num{0.509261} & --- \\
    {$\Delta t=T/256$}
    &\num{0.25401}	
    &\numbf{0.192922}
    &\num{0.173274}	
    &\num{0.167964} & \numQ{1.60}\\
    {$\Delta t=T/1024$}
    &\num{0.198474}
    &\num{0.107551}	
    &\numbf{0.065671}
    &\num{0.0500585} &\numQ{1.74} \\
    {$\Delta t=T/4096$}
    &\num{0.193175}
    &\num{0.0982567}
    &\num{0.0498987}
    &\numbf{0.0266914} & \numQ{0.907276}\\
    \midrule
     \multicolumn{1}{r}{\eoc{x}}
   & {---}
    &\numQ{0.975276}
    &\numQ{0.977549}
    &\numQ{0.90265755}\\
    \midrule
     \multicolumn{1}{r}{\underline{\eoc{xtt}}}
   & {---}
    &\numQ{1.4664724}
    &\numQ{1.5546893}
    &\numQ{1.298903}\\
    \bottomrule
  \end{tabular}

  \begin{tabular}{@{~}l@{~~}r@{~~}r@{~~}r@{~~}r@{~}r@{~}}
    \toprule
    \multicolumn{5}{c}{$L^{\infty}(L^2)$-norm of the error} \\
    \midrule
    & $h=1/2$
    & $h=1/4$
    & $h=1/8$
    & $h=1/16$ & \eoc{tt} \\
    \midrule
    $\Delta t=T/64$
    &\numbf{0.841768}
    &\num{0.853377}
    &\num{0.856291}
    &\num{0.857016} & -- \\
    $\Delta t=T/256$
    &\num{0.238577}
    &\numbf{0.240204}
    &\num{0.242982}
    &\num{0.243812} & \numQ{1.813553}\\
    $\Delta t=T/1024$
    &\num{0.0949171}
    &\num{0.06082}
    &\numbf{0.0607511}
    &\num{0.0614854} & \numQ{1.9874627}\\
    $\Delta t=T/4096$
    &\num{0.0880713}
    &\num{0.025167}
    &\num{0.01528}
    &\numbf{0.0152065} & \numQ{2.0154989}\\
    \midrule
    \multicolumn{1}{r}{\eoc{x}}
    &{---}
    &\numQ{1.807134}
    &\numQ{0.7198887}
    &\numQ{0.00690897}\\
    \midrule
    \multicolumn{1}{r}{\underline{\eoc{xtt}}}
    &{---}
    &\numQ{1.8091624}
    &\numQ{1.98328}
    &\numQ{1.99817258}\\
    \bottomrule
  \end{tabular}

\end{center}
\vspace*{-0.25cm}
\end{table}

{\bf Example 2.}  We now consider the Allen-Cahn equation on a sphere of varying radius $R(t)$. The level set function of the sphere is given by
$$
\phi=x^2+y^2+z^2-R(t)^2.
$$
It defines a pulsation of the sphere. We are interested if the numerical solution approximation a geodesic curvature type flow defined by \eqref{e:dynDistn}. The phase separation curve $C(t)$  is initially a circle with radius $r_0<R(0)$. Due
to the axial symmetry, for all $t\in[0,t_{crit})$, $C(t)$ is a circle of radius $r(t)$,
 where  $r(t)$ solves the ODE (cf. Appendix~\ref{A1})
\begin{equation}\label{e:ODE}
r_t=\frac{r^2-R^2}{rR^2}+\frac{r}{R}R_t.
\end{equation}

Our reference solution is computed by the direct integration of   \eqref{e:ODE} with a higher order Runge--Kutta method.
We next solve
the Allen-Cahn equation on the sphere and compare the radius of the zero level-set
of the numerical solution with the reference solution.  In this test, we set
$R(t)=\frac{1}{\sqrt{1+\delta\cos nt }}$, with $\delta=\frac{1}{6}$ and $n=16\pi$.
We choose the final time $T=0.125$ and $\delta t\approx 3.9063\times 10^{-6}$.
{\color{black}We set $u_0=\mathrm{tanh}(d_{C_0}(x)/\varepsilon)$ where $d_{C_0}(x)$ is the signed (geodesic) distance function to the circle $C(0)$
with radius $r_0$ on the initial sphere. We compute the numerical solution for several values of  $\varepsilon=0.4,0.2,0.1$.}
The {\color{black} averaged } radius evolution recovered from the
finite element solution to the Allen--Cahn equation is shown in Figure~\ref{fig:meanCurvature2}.
We can see that results are  in a good agreement with the reference solution, {\color{black} and converge to the true solution for decreasing $\eps$.}

\begin{figure}[ht!]
 \begin{subfigure}{0.6\linewidth}
 \begin{overpic}[width=\textwidth]{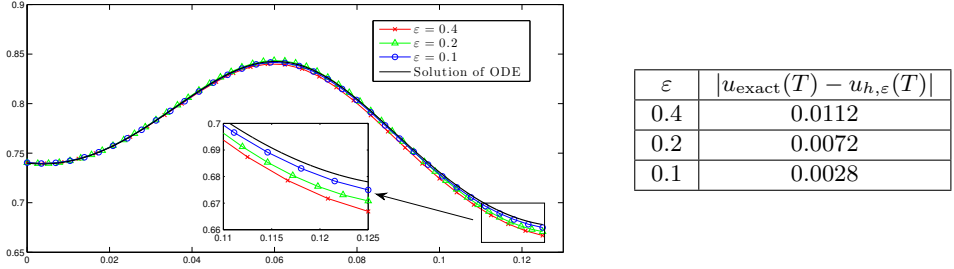}
  \end{overpic}
  \end{subfigure}
  \begin{subfigure}{0.39\linewidth} {\color{black}
		\centering\small
		\begin{tabular}[1.2]{|c|c|}
			\hline
			$\eps$ & $|u_{\rm exact}(T)-u_{h,\eps}(T)|$  \\
			\hline
			0.4 & 0.0112 \\
			\hline
			0.2 & 0.0072 \\
			\hline
			0.1  & 0.0028 \\
			\hline
		\end{tabular}
}
	\end{subfigure}%
 \caption{Example~2: Approximation to a mean curvature flow for varying $\eps$. The error between the true and numerical solution is shown at final time $T=0.125$.}\label{fig:meanCurvature2}
\end{figure}

{\bf Example 3.} In this example, we consider the surface Allen-Cahn equation~\eqref{transport}  on a deforming manifold of a general shape.
The initial manifold is given ({as in \cite{Dziuk88}}) by
$$
\Gamma(0)=\{ \, \bx\in \Bbb{R}^3~|~ (x -z^2)^2+y^2+z^2=1\, \}
$$
The velocity field that deforms the surface is
$$\mathbf{w}(\bx,t)=\big(10 x \cos(100t),20 y \sin(100t),20 z \cos(100t)\big)^T.$$
In this example, we choose a slightly different $f(u)=u^2(1-u^2)$ so that solution is in the interval $[0,1]$.  The initial function $u_0$ is defined in each node by a random number from $[0,1]$ using the uniform distribution.

In this example, we set $a=1$, $\eps=0.01$, $T=0.04$ and $\Omega=[-2,2]^3$. We use the same bulk triangulation and spaces as in example 1 and  $\Delta t=T/1024$.
Figure~\ref{fig:dziuk} shows  the (approximated) manifold and snapshots of the discrete solution $u_h$ at several time instances. In general, we note that the evolution of $u$ in this example is similar to what is found on the  stationary surface $\Gamma(0)$ with surface FEM in \cite{dziuk2007surface}: the fast decomposition phase follows by the formation of phases with a narrow transition region (diffuse interface) between phases. As expected for the mean curvature motion, the  interface tends to straightening, second phase regions are rounding and shrinking.

\begin{figure}[ht!]
 \centering

   \vspace{1mm}
 \begin{overpic}[width=0.45\textwidth]{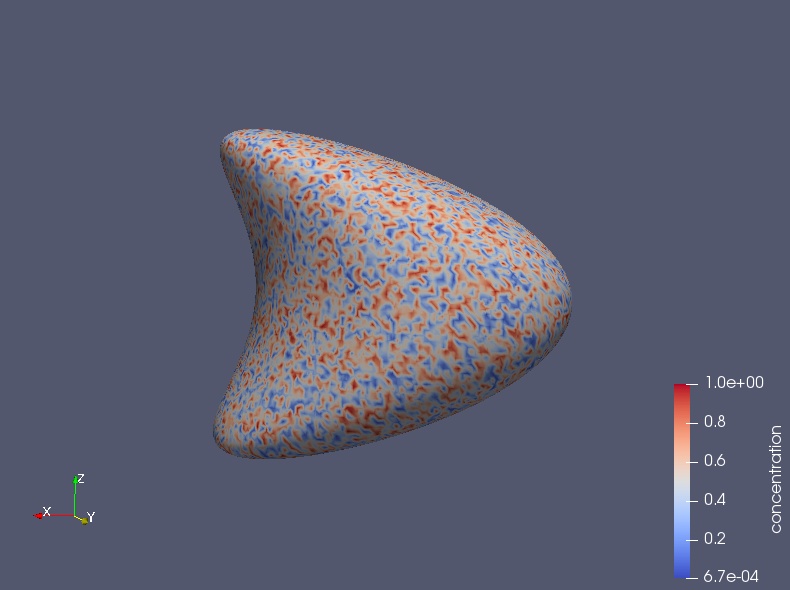}
  \end{overpic}
   \begin{overpic}[width=0.45\textwidth]{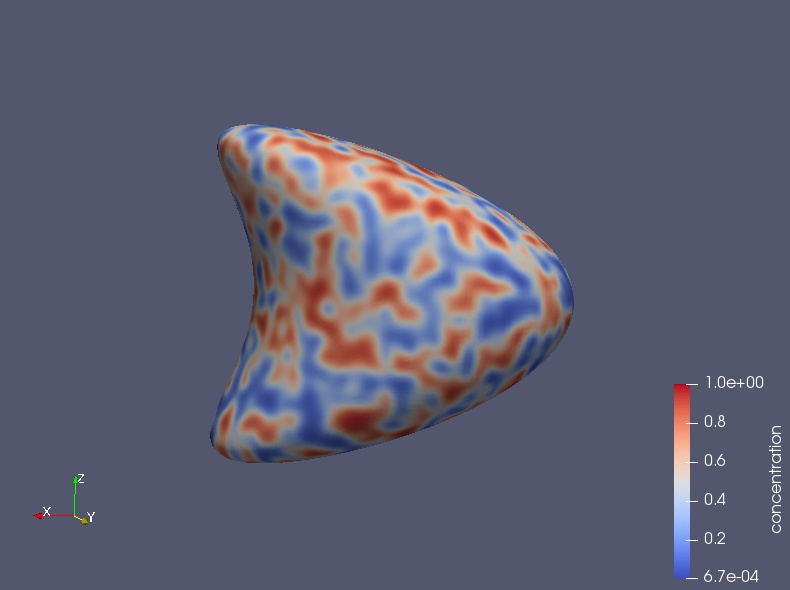}
  \end{overpic}
  \vspace{1mm}

   \begin{overpic}[width=0.45\textwidth]{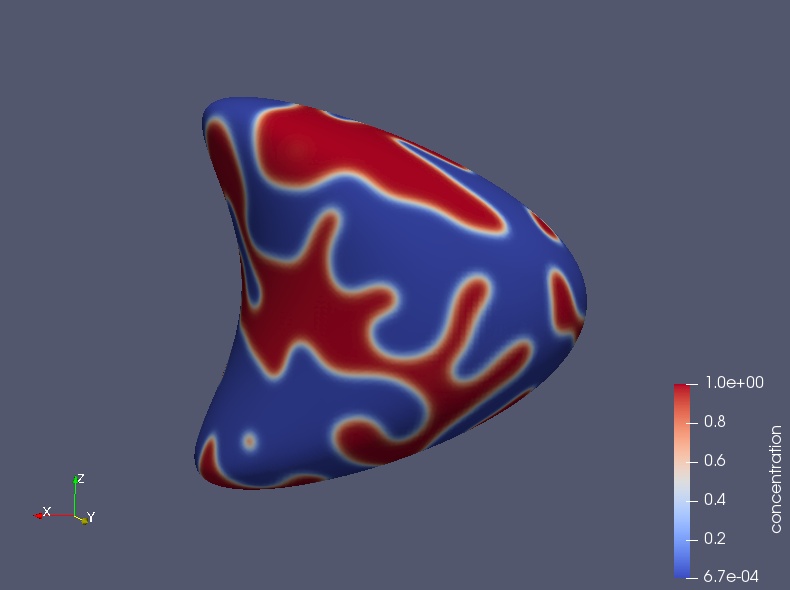}
  \end{overpic}
   \begin{overpic}[width=0.45\textwidth]{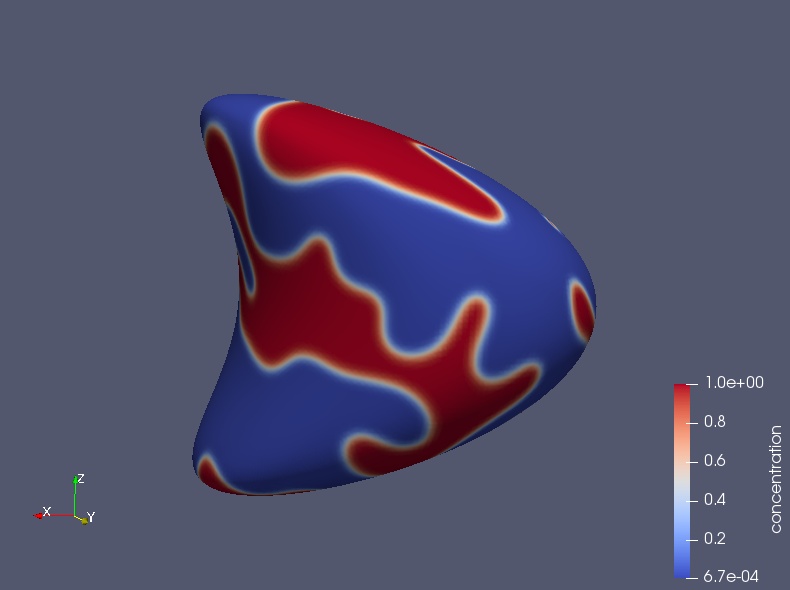}
  \end{overpic}

  \vspace{1mm}
   \begin{overpic}[width=0.45\textwidth]{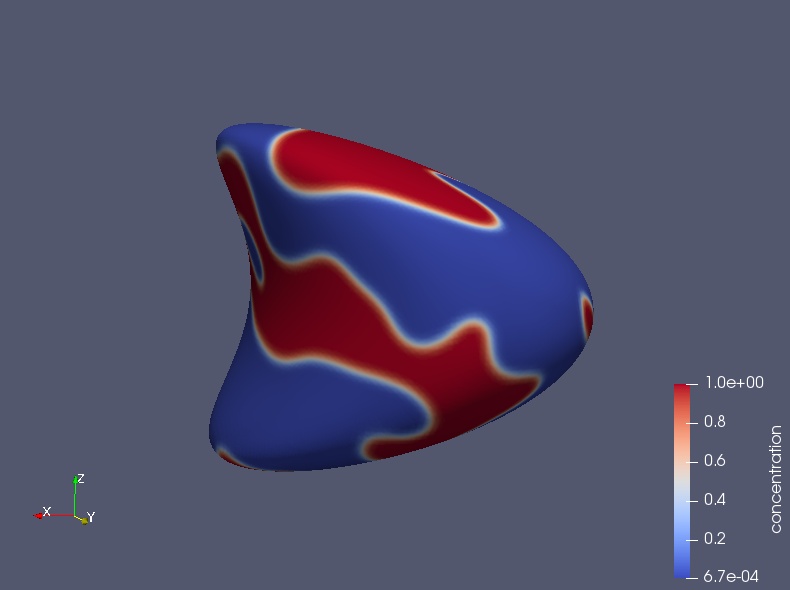}
  \end{overpic}
   \begin{overpic}[width=0.45\textwidth]{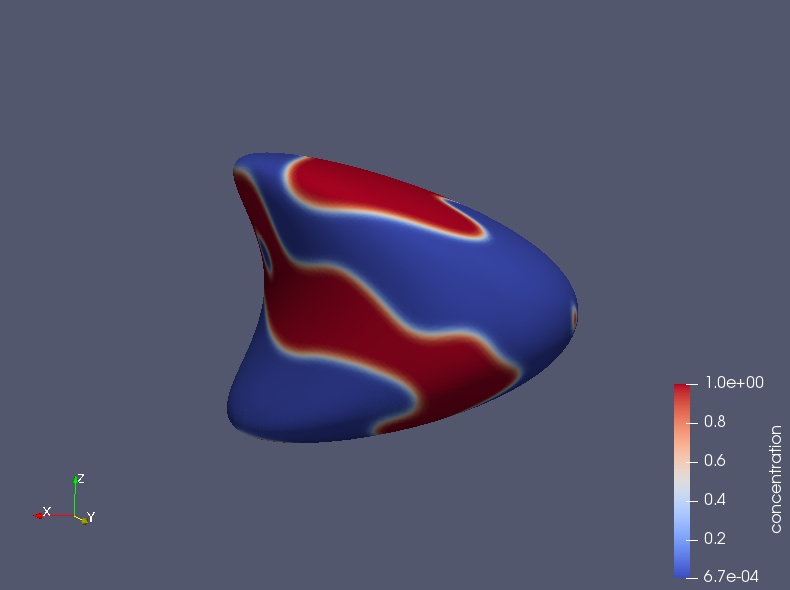}
  \end{overpic}
 \caption{Example~3: solutions for $t=k\Delta t$ with $k=0,32,256,512,768,1024$.}\label{fig:dziuk}
\end{figure}

\section*{Appendix} \label{A1}
{\color{black}
\begin{figure}[ht!]
 \centering
 \begin{overpic}[width=0.35\textwidth]{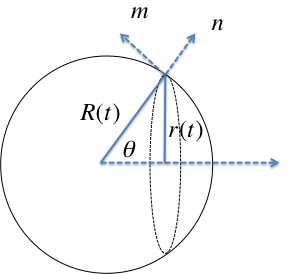}
  \end{overpic}
 \caption{Illustration of quantities in \eqref{vel}.}\label{fig:circle}
\end{figure}
We give a brief derivation of \eqref{e:ODE}.
On a sphere of a varying radius $R(t)$  consider a  circle $C(t)$ of radius $r(t)$ (see Figure~\ref{fig:circle}).
Assume the circle evolves according to the geodesic curvature flow given by \eqref{e:dynDistn}.
The geodesic curvature can be computed as the curvature of the circle projection on the tangential planes:
$$
\kappa_g=\frac{1}{r}\cos\theta=\frac{\sqrt{R^2-r^2}}{rR}.
$$
This determines the co-normal velocity of $C(t)$, while the normal velocity is given by ${R}_t$.
Therefore, the material velocity of the points on $C(t)$ is given by
\begin{equation}\label{vel}
-\frac{\sqrt{R^2-r^2}}{rR} \mathbf{m}+{R}_t \mathbf{n}.
\end{equation}
Then the time derivative of the radius $r$ can be explicitly computed to be
$$
r_t= -\frac{\sqrt{R^2-r^2}}{rR} \cos\theta+{R}_t \sin\theta=\frac{r^2-R^2}{r R^2}+\frac{r}{R}R_t.
$$
}

\section*{Acknowledgement}
{X.X. acknowledges the financial support by NSFC project under Grant 11971469 and by the National Key R\&D Program of China under Grant 2018YFB0704304 and Grant 2018YFB0704300}.
M.O. was partially supported by  NSF through the Division of Mathematical Sciences grants DMS-2011444 and DMS/NIGMS-1953535.

\vspace*{-0.2cm}

\bibliography{literatur}{}
\bibliographystyle{siam}
\end{document}